\documentclass[reqno]{amsart}

\usepackage{graphicx,mathtools,enumitem,xcolor}
\usepackage{amssymb}

\theoremstyle{plain}
\newtheorem{lemma}{Lemma}
\newtheorem{theorem}[lemma]{Theorem}

\newtheorem{definition}[lemma]{Definition}

\theoremstyle{remark}
\newtheorem{remark}{Remark}
\newtheorem{case}{Case}

\newcommand {\eps}{\varepsilon}

\newcommand  {\N}{\mathbb{N}}
\newcommand  {\R}{\mathbb{R}}
\newcommand  {\I}{\mathbb{I}}

\renewcommand{\d}{\mathrm{d}}
\newcommand  {\e}{\mathrm{e}}

\newcommand  {\Ic}{\mathcal{I}}
\newcommand  {\Jc}{\mathcal{J}}

\newcommand  {\HK}{K}
\newcommand{\erfc}{\operatorname{erfc}}
\newcommand{\erf}{\operatorname{erf}}

\renewcommand{\l}{{\operatorname{l}}}
\renewcommand{\r}{{\operatorname{r}}}
\newcommand{\full}{{\operatorname{full}}}
\newcommand{\loc}{{\operatorname{loc}}}

\DeclareMathOperator*{\esssup}{ess\,sup}
\DeclareMathOperator*{\essinf}{ess\,inf}
\DeclareMathOperator*{\supp}{supp}

\newcommand*{\abs}[1]{\lvert #1 \rvert}

\newcommand{\NewVersionOff}[1]{}

\begin{document}
\title[Long-time asymptotics of solutions to the Keller--Rubinow model]%
{Long-time asymptotics of solutions to the Keller--Rubinow model
  for Liesegang rings \\ in the fast reaction limit}

\author[Z. Darbenas]{Zymantas Darbenas}
\author[R. v.~d.~Hout]{Rein van der Hout}
\author[M. Oliver]{Marcel Oliver}
\email[M. Oliver]{m.oliver@jacobs-university.de}

\address[Z. Darbenas and M. Oliver]%
{School of Engineering and Science \\
 Jacobs University \\
 28759 Bremen \\
 Germany}

\address[R. v.~d.~Hout]%
{Dunolaan 39 \\
6869VB Heveadorp \\
The Netherlands}

\date{\today}

\begin{abstract}
We consider the Keller--Rubinow model for Liesegang rings in one
spatial dimension in the fast reaction limit as introduced by
Hilhorst, van der Hout, Mimura, and Ohnishi in 2007.  Numerical
evidence suggests that solutions to this model converge, 
independent of the initial concentration, to a universal profile for large times
in parabolic similarity coordinates.  For the concentration function,
the notion of convergence appears to be similar to attraction to a
stable equilibrium point in phase space.  The reaction term, however,
is discontinuous so that it can only convergence in a much weaker,
averaged sense.  This also means that most of the traditional
analytical tools for studying the long-time behavior fail on this
problem.

In this paper, we identify the candidate limit profile as the solution
of a certain one-dimensional boundary value problem which can be
solved explicitly.  We distinguish two nontrivial regimes.  In the
first, the \emph{transitional regime}, precipitation is restricted to
a bounded region in space.  We prove that the concentration converges
to a single asymptotic profile.  In the second, the
\emph{supercritical regime}, we show that the concentration converges
to one of a one-parameter family of asymptotic profiles, selected by a
solvability condition for the one-dimensional boundary value problem.
Here, our convergence result is only conditional: we prove that if
convergence happens, either pointwise for the concentration or in an
averaged sense for the precipitation function, then the other field
converges likewise; convergence in concentration is uniform, and the
asymptotic profile is indeed the profile selected by the solvability
condition.  A careful numerical study suggests that the actual
behavior of the equation is indeed the one suggested by the theorem.
\end{abstract}

\maketitle

\tableofcontents

\section{Introduction}

Liesegang rings appear as regular patterns in many chemical
precipitation reactions.  Their discovery is usually attributed to the
German chemist Raphael Liesegang who, in 1896, observed the emergence
of concentric rings of silver dichromate precipitate in a gel of
potassium dichromate when seeded with a drop of silver nitrate
solution.  Related precipitation patterns were in fact observed even
earlier, see \cite{Henisch:1988:CrystalsGL} for a historical
perspective.

{From} the modeling perspective, there are two competing points of
view.  One is a ``post-nucleation'' approach in which the patterns
emerge via competitive growth of precipitation germs
\cite{Smith:1984:OstwaldST}, the other a ``pre-nucleation'' approach,
a sophisticated modification of the ``post-nucleation'' approach,
suggested by Keller and Rubinow \cite{KellerR:1981:RecurrentPL} which
is the starting point of the present work.  The recent
survey~\cite{DuleyFM:2017:KellerRM} gives a comprehensive summary of
the most important published research on both approaches, including
numerical and theoretical comparisons.  A direct and detailed
comparison between the two theories and the history behind can be
found in~\cite{KrugB:1999:MorphologicalCL}.

The Keller--Rubinow model is based on the chain of chemical reactions
\begin{gather*}
  A + B \to C \to D
\end{gather*}
with associated reaction-diffusion equations
\begin{subequations}
\begin{gather}
  a_t = \nu_a \, \Delta a - k \, a \, b \,, \\
  b_t = \nu_b \, \Delta b - k \, a \, b \,, \\
  c_t = \nu_c \, \Delta c + k \, a \, b - P(c,d) \,, \\
  d_t = P(c,d) \,,
\end{gather}
\end{subequations}
where the rate of the precipitation reaction is described by the
function
\begin{gather}
  P(c,d) =
  \begin{cases}
    0 & \text{if } d=0 \text{ and } c<c^\top \,, \\
    \lambda \, (c-c^\bot)_+ &
    \text{if } d>0 \text{ or } c \geq c^\top \,.
  \end{cases} 
\end{gather}
Without loss of generality, we may assume that the precipitation rate
constant $\lambda=1$; this choice is assumed in the remainder of the
paper.  The precipitation function $P$ expresses that precipitation
starts only once the concentration $c$ exceeds a supersaturation
threshold $c^\top$ and continues for as long as $c$ exceeds the
saturation threshold $c_\bot$.

Using
\cite{HilhorstHP:1996:FastRL,HilhorstHP:1997:DiffusionPF,HilhorstHP:2000:NonlinearDP},
Hilhorst \emph{et al.}\
\cite{HilhorstHM:2007:FastRL,HilhorstHM:2009:MathematicalSO} studied
the case where $\nu_b=0$, $c_\bot=0$, and the ``fast reaction limit''
where $k \to \infty$.  To simplify matters, they took as the spatial
domain the positive half-axis.  This is precisely the setting we shall
consider in our work and which we refer to as the HHMO-model.  Writing
$u$ in place of $c$ and choosing dimensions in which $\nu_c=1$, we can
state the model as
\begin{subequations}
  \label{e.original}
\begin{gather}
  u_t = u_{xx} +
        \frac{\alpha \beta}{2 \sqrt t} \, \delta (x - \alpha \sqrt{t})
        - p[x,t;u] \, u \,,
  \label{e.original.a} \\
  u_x(0,t) = 0 \quad \text{for } t \geq 0 \,, \label{e.original.b} \\
  u(x,0) = 0 \quad \text{for } x>0 \label{e.original.c}
\end{gather}
\end{subequations}
where the precipitation function $p[x,t;u]$ depends on $x$, $t$, and
nonlocally on $u$ via
\begin{equation}
  p[x,t;u] = H
  \biggl(
    \int_0^t (u(x,\tau) - u^*)_+ \, \d \tau
  \biggr) \,.
  \label{e.hhmo-p}
\end{equation}
Here, $H$ denotes the Heaviside function with the convention that
$H(0)=0$ and $u^*$ denotes the supersaturation concentration.

Hilhorst \emph{et al.}\ \cite{HilhorstHM:2007:FastRL} further
introduce the notion of a weak solution to \eqref{e.original}.  Modulo
technical details, their approach is to seek pairs $(u,p)$ that
satisfy \eqref{e.original.a} integrated against a suitable test
function such that
\begin{equation}
\label{e.hhmo-p-weak}
  p(x,t) \in H
  \biggl(
    \int_0^t (u(x,\tau) - u^*)_+ \, \d \tau
  \biggr)
\end{equation}
where $H$ now denotes the Heaviside graph
\begin{equation}
\label{p.H.def}
  H(y) \in
  \begin{cases}
    0 & \text{when } y<0 \,, \\
    [0,1] & \text{when }y=0 \,, \\
    1 & \text{when } y>0 \,.
  \end{cases}
\end{equation}
Additionally, they require that $p(x,t)$ takes the value $0$ whenever
$u(x,s)$ is strictly less than the threshold $u^*$ for all
$s\in[0,t]$.  This can be stated as
\begin{equation}
\label{e.hhmo-p-weak-alternative}
  p(x,t)\in
  \begin{cases}
     0&\text{ if }\sup_{s\in[0,t]}u(x,s)<u^* \,,\\
     [0,1]&\text{ if }\sup_{s\in[0,t]}u(x,s)=u^* \,,\\
     1 &\text{ if }\sup_{s\in[0,t]}u(x,s)>u^* \,.
  \end{cases}
\end{equation}
The question of uniqueness of weak solutions is open in general.
However, in \cite{DarbenasHO:2020:ConditionalUS} we prove short-time
uniqueness and show that solutions remain unique so long as a certain
transversality condition is satisfied.  Further,
\cite{HilhorstHM:2009:MathematicalSO} pose the question whether the
precipitation function $p$ is binary.  Rigorous results on a
simplified system as well as numerics indicate that solutions with a
binary precipitation function only exist over a finite interval of
time \cite{Darbenas:2018:PhDThesis, DarbenasO:2019:UniquenessSW,
DarbenasO:2021:BreakdownLP}.

In this paper, we provide evidence that the long-time behavior of
solutions to the HHMO-model is determined by an asymptotic profile
that depends only on the parameters of the equation.  Heuristically,
the mechanism of convergence is the following: as soon as the
concentration exceeds the precipitation threshold $u^*$, the reaction
ignites and reduces the reactant concentration.  A continuing reaction
burns up enough fuel in its neighborhood to eventually pull the
concentration below the threshold everywhere, so the reaction region
cannot grow further.  Eventually, the source location will move
sufficiently far from the active reaction regions that the
concentration grows again and the reaction threshold may be surpassed
again.  As the source loses strength with time, the amplitudes of the
concentration change around the source will decrease with time,
getting ever closer to the critical concentration.  In fact, both
numerical studies and analytical results on a simplified model suggest
that convergence of concentration to the critical value happens within
a bounded region of space-time \cite{DarbenasO:2021:BreakdownLP}, so
that process of equilibrization is much more rapid than the typical
approach to a stable equilibrium point in a smooth dynamical system.

In $x$-$t$ coordinates, the source point is moving.  To analyze the
time-asymptotic behavior, we must therefore change into parabolic
similarity coordinates, here defined as $\eta=x/\sqrt{t}$ and
$s=\sqrt{t}$.  We further write $u(x,t) = v(x/\sqrt{t},\sqrt{t})$ and
$p[x,t;u] = q[\eta,s;v]$ to make transparent which coordinate system
is used at any point in the paper.  In similarity coordinates, the
$\delta$-source in \eqref{e.original} is stationary at $\eta=\alpha$
but decreases in strength as time progresses.  In what follows, we
look for asymptotic profiles where
\begin{equation}
  \lim_{s \to \infty} v(\alpha, s) = u^* \,.
  \label{e.valphalimit}
\end{equation}
In the classical setting of smooth dynamical systems, the limit
function would correspond to a stable equilibrium of the system in
$\eta$-$s$ coordinates.  Here, stationarity is incompatible with the
ignition condition \eqref{e.hhmo-p-weak-alternative}.  We thus impose
that $p$ takes a form such that the precipitation term loses its
$s$-dependence.  This requirement can only be satisfied when
$p(x,t) = \gamma \, x^{-2} \, H(\alpha^2t-x^2)$ for some non-negative
constant $\gamma$, so the self-similar precipitation function takes
values outside of $[0,1]$, in fact, it is not even bounded.
Nonetheless, for each $\gamma \geq 0$, we can solve the stationary
problem to obtain a profile $\Phi_\gamma$, which, subject to suitable
conditions, is uniquely determined by the condition that
$\Phi_\gamma(\alpha) = u^*$ so the profile is consistent with the
conjectured limit \eqref{e.valphalimit}.  Now the following picture
emerges.

With varying source strength (in the following, we will actually think
of varying $u^*$ for given values of $\alpha$ and $\beta$), there are
three distinct open regimes.  When the source is insufficient to
ignite the reaction at all (``subcritical regime''), the dynamics
remains trivial.  When the source strength is larger but not very
large (``transitional regime''), some reaction will be triggered
initially, but eventually diffusion overwhelms the source so that no
further ignition occurs.  The scenario of asymptotic equilibrization
cannot be maintained so that \eqref{e.valphalimit} does not hold true.
We find that solutions anywhere in the transitional regime will
converge to a universal profile $\Phi_0$.  When the source strength is
large enough so that continuing re-ignition is always possible
(``supercritical regime''), we identify a one-parameter family of
profiles $\Phi_\gamma$ which determine the long-time asymptotics of
the concentration; in particular, \eqref{e.valphalimit} holds true.

Throughout the paper, we use the following notion of convergence.  For
the concentration, we look at uniform convergence in $\eta$-$s$
coordinates, i.e.,
\begin{equation}
  \lim_{s \to \infty} \sup_{\eta \geq 0} \,
    \lvert v(\eta,s) - \Phi_\gamma(\eta) \rvert
  = \lim_{t \to \infty} \sup_{\eta \geq 0} \,
    \lvert u(\eta \sqrt t, t) - \Phi_\gamma(\eta) \rvert
  = 0 \,.
  \label{e.uniform}
\end{equation}
For brevity, we shall say that \emph{$u$ converges uniformly to
$\Phi_\gamma$}; the sense of convergence is always understood as
defined here.

For the precipitation function, the notion of convergence is more
subtle.  For our main results, we work with precipitation functions
that satisfy the following condition:
\begin{itemize}
\item [(P)] There exists a measurable function $p^*$ such that for
a.e.\ $x \in \R_+$,
\begin{equation}
  \label{p.property}
  p(x,t) = p^*(x) \quad \text{for} \quad t > x^2/\alpha^2 \,.
\end{equation}
\end{itemize}
Condition (P) expresses that there is no ignition of precipitation in
the region $\eta<\alpha$.  When the concentration passes the threshold
transversally, this condition is always satisfied.  When the
concentration reaches, but does not exceed the threshold on sets of
positive measure, weak solutions in the sense of
\cite{HilhorstHM:2009:MathematicalSO} may violate (P).  Thus,
condition (P) provides a selection criterion to distinguish physical
from unphysical weak solutions.  In
\cite{DarbenasHO:2020:ConditionalUS}, we show that a minor
modification of the construction in
\cite{HilhorstHM:2009:MathematicalSO} proves existence of weak
solutions that satisfy condition (P).  Thus, in hindsight, it would be
most natural to incorporate (P) into the definition of weak solutions
from the start.  However, to remain closer to the prior literature and
also to make more transparent which parts of the argument depend on
(P), we carry (P) as a separate condition throughout.

Referring to (P), we can define a notion of convergence for the
precipitation function; it is
\begin{equation}
  \lim_{x \to \infty} x \int_x^\infty p^*(\xi) \, \d \xi = \gamma \,.
  \label{e.p-convergence}
\end{equation}
This means that, in an integral sense, the precipitation function
along the line $\eta = \alpha$ has the same long-time asymptotics as
the precipitation function of the self-similar profile, where
$p^*(x) = \gamma/x^2$.

The results in this paper are the following.  First, we derive an
explicit expression for $\Phi_\gamma$ and prove necessary and
sufficient conditions under which it is a solution to the stationary
problem with self-similar precipitation function.  Second, we present
numerical evidence that the solution indeed converges to the
stationary profile as described.  Third, we prove that $\Phi_0$ is the
stationary profile in the transitional regime.  Fourth, in the
supercritical regime, we can only give a partial result which states
the following: If there is an asymptotic profile for the
HHMO-solution, it must be $\Phi_\gamma$ and the precipitation function
$p$ is asymptotic to the self-similar profile in the sense of
\eqref{e.p-convergence}.  Vice versa, if the precipitation function is
asymptotic to the self-similar profile, then it also satisfies
\begin{equation}
  \lim_{x \to \infty} \frac1x \int_0^x \xi^2 \, p^*(\xi) \, \d \xi
  = \gamma
\end{equation}
and the concentration $u$ converges uniformly to the profile
$\Phi_\gamma$.

The main remaining open problem is the proof of unconditional
convergence to the self-similar profile.  Part of the difficulty is
that the asymptotic behavior of the precipitation function described
above is non-local in time.  Thus, it is not clear how to pass from
convergence on a subsequence (for example, convergence of the time
average of the concentration is easily obtained via a standard
compactness argument) to convergence in general.  There is a second,
more general open question.  The derivation of the only compatible
asymptotic profile might generalize to a procedure for coarse-graining
dynamical systems whose microscopic dynamics consists of strongly
equilibrizing switches as we find in the HHMO-model for Liesegang
rings.  A precise understanding of the necessary conditions, however,
remains wide open.

Let us explain how our work relates to the extensive literature on
relay hysteresis.  The precipitation condition can be seen as a
\emph{non-ideal relay} with switching levels $0$ and $u^*$.  Its
generalization to non-binary values for $p$ in \eqref{p.H.def} or
\eqref{e.hhmo-p-weak-alternative} can be seen as a \emph{completed
relay} in the sense of Visintin
\cite{Visintin:1986:EvolutionPH,Visintin:1994:DifferentialMH}, see
also Remark~\ref{r.monotonicity}.  Local well-posedness of a
reaction-diffusion equation with a non-ideal relay reaction term was
proved by Gurevich \emph{et al.}\ \cite{GurevichST:2013:ReactionDE}
subject to a transversality condition on the initial data.  If this
condition is violated, the solution may be continued only in the sense
of a completed relay, where existence of solutions is shown in
\cite{Visintin:1986:EvolutionPH,AikiK:2008:MathematicalMB}, but
uniqueness is generally open.  Gurevich and Tikhomirov
\cite{GurevichT:2017:RattlingSD,GurevichT:2018:SpatiallyDR} show that
a spatially discrete reaction-diffusion system with relay-hysteresis
exhibits ``rattling,'' grid-scale patterns of the relay state which
are only stable in the sense of a density function.  The question of
optimal regularity of solutions to reaction-diffusion models with
relay hysteresis is discussed in \cite{ApushkinskayaU:2015:FreeBP}.
For an overview of recent developments in the field, see
\cite{CurranGT:2016:RecentAR,Visintin:2014:TenIH}.

The study of the HHMO-model as introduced above shares many features
with the results in the references cited above; it is also marred by
the same difficulties.  However, there is also a key difference to the
systems studied elsewhere: the source term in the HHMO-model is local
and, reflecting its origin through a fast-reaction limit, follows
parabolically self-similar scaling.  Thus, the nontrivial dynamics
comes from the interplay of the parabolic scaling in the forcing and
the memory of the reaction term which is attached to locations $x$ in
physical space.  The parabolic scaling also necessitates studying the
system on an unbounded domain, even though, in practice, the
concentration is rapidly decaying and can be well-approximated on
bounded domains, see Section~\ref{s.numerics} and
Appendix~\ref{a.scheme} below.  The HHMO-model has enough symmetries
that a study of the long-time behavior of the solution is possible; we
are not aware of corresponding results for other reaction-diffusion
equations with relay-hysteresis.

The paper is structured as follows.  In the preliminary
Section~\ref{s.without}, we rewrite the equations in standard
parabolic similarity variables and derive the similarity solution
without precipitation, which is a prerequisite for defining the notion
of weak solution and is also used as a supersolution in several
proofs.  In Section~\ref{weak.solution}, we recall the concept of weak
solution from \cite{HilhorstHM:2009:MathematicalSO} and prove several
elementary properties which follow directly from the definition.  In
Section~\ref{s.self-similar}, we introduce the self-similar
precipitation function, derive the stationary solution in similarity
variables and prove necessary and sufficient conditions for their
existence under the required boundary conditions.
Section~\ref{s.numerics} describes the phenomenology of solutions to
the HHMO-model by numerical simulations which confirm the picture
outlined above; details about the numerical code are given in the
appendix.  The final two sections are devoted to proving rigorous
results on the long-time asymptotics.  In Section~\ref{s.auxiliary},
we study the long-time dynamics of a linear auxiliary problem, in
Section~\ref{s.hhmo} we use the results on the auxiliary problem to
state and prove our main theorems on the long-time behavior of the
HHMO-model.

\section{Self-similar solution without precipitation}
\label{s.without}

For the reader's convenience, we recall the derivation of the
self-similar solution to the model without precipitation which was
already introduced in
\cite{HilhorstHM:2007:FastRL,HilhorstHM:2009:MathematicalSO} and is
required to define the notion of weak solution for the full model in
the next section.

Writing \eqref{e.original} in terms of the parabolic similarity
coordinates $\eta=x/\sqrt{t}$ and $s=\sqrt{t}$, and setting
$u(x,t) = v(x/\sqrt{t},\sqrt{t})$, $p[x,t;u] = q[\eta,s;v] \equiv q$
and
$\delta(\eta - \alpha) = \delta_\alpha(\eta) \equiv \delta_\alpha$, we
obtain
\begin{subequations}
  \label{e.v}
\begin{gather}
  s \, v_s - \eta \, v_\eta = 2 \, v_{\eta\eta}
  + \alpha \beta \, \delta_\alpha - 2 \, s^2 \, q[\eta,s;v] \, v \,,
  \label{e.v-a} \\
  v_\eta(0,s) = 0 \quad \text{for } s > 0 \,. \label{e.v-c}
\end{gather}
\end{subequations}
Since the change of variables is singular at $s=0$, we cannot
translate the initial condition \eqref{e.original.c} into $\eta$-$s$
coordinates.  We will augment system \eqref{e.v} with suitable
conditions when necessary.

Self-similar solution are steady-states in $\eta$-$s$ coordinates.  We
first consider the case where $p=0$ or $q=0$, respectively.  Then
\eqref{e.v} reduces to the ordinary differential equation
\begin{subequations}
  \label{e.v-zerop}
\begin{gather}
  \Psi'' + \frac\eta2 \, \Psi' + \frac{\alpha\beta}2 \, \delta(\eta-\alpha)
  = 0 \,, \label{e.ode1} \\
  \Psi'(0) = 0 \,, \label{e.v-zerop-b} \\
  \Psi(\eta) \to 0 \quad \text{as } \eta \to \infty \,.
  \label{e.v-zerop-c}
\end{gather}
\end{subequations}
Condition \eqref{e.v-zerop-c} encodes that we seek solutions where the
total amount of reactant is finite.  Note that in the full
time-dependent problem, decay of the solution at spatial infinity is
encoded into the initial data and must be shown to propagate in time
within an applicable function space setting.

The integrating factor for \eqref{e.ode1} is $\exp(\tfrac14 \eta^2)$,
so that by integrating with respect to $\eta$ and using
\eqref{e.v-zerop-b} as initial condition, we find
\begin{equation}
  \Psi'(\eta) = - \frac{\alpha\beta}2 \,
    \e^{\tfrac{\alpha^2-\eta^2}4} \, H(\eta-\alpha) \,.
\end{equation}
Another integration, this time on the interval $[\eta,\infty)$ using
condition \eqref{e.v-zerop-c}, yields
\begin{align}
  \label{self.similar.p0}
  \Psi(\eta)
  & = \frac{\alpha\beta}2 \, \e^{\tfrac{\alpha^2}4}
      \int_\eta^\infty \e^{-\tfrac{\zeta^2}4} \,
        H(\zeta-\alpha) \, \d \zeta
      \notag \\
  & = \frac{\alpha \beta \sqrt\pi}2 \, \e^{\tfrac{\alpha^2}4} \cdot
      \begin{cases}
        \erfc(\alpha/2) & \text{if } \eta \leq \alpha \,, \\
        \erfc(\eta/2) & \text{if } \eta > \alpha \,.
      \end{cases}
\end{align}
Translating this result back into $x$-$t$ coordinates and setting
$\psi(x,t)=\Psi(x/\sqrt t)$, we obtain the self-similar,
zero-precipitation solution,
\begin{gather}
\label{psi.def}
  \psi(x,t)
  = \begin{dcases}
      \frac{\alpha\beta}2 \, \e^{\tfrac{\alpha^2}4}
      \int_{\alpha\vphantom\int}^\infty
      \e^{-\tfrac{\zeta^2}4} \, \d \zeta &
      \text{if } x \leq \alpha \sqrt{t} \,, \\
      \frac{\alpha\beta}2 \, \e^{\tfrac{\alpha^2}4}
      \int_{x/\sqrt{t}}^{\infty\vphantom\int}
      \e^{-\tfrac{\zeta^2}4} \, \d \zeta &
      \text{if } x > \alpha \sqrt{t} \,.
    \end{dcases}
\end{gather}

\section{Weak solutions for the HHMO-model}
\label{weak.solution}

We start with a rigorous definition of a (weak) solution for the
HHMO-model \eqref{e.original}.  In this formulation, we allow for
fractional values of the precipitation function $p$ as \emph{a priori}
we do not know whether $p$ is binary, or will remain binary for all
times.

For non-negative integers $n$ and $k$, and $D \subset \R\times\R_+$
open, we write $C(D)$ to denote the set of continuous real-valued
functions on $D$, and
\begin{subequations}
\begin{gather}
   C^{n,k}(D)
  = \Bigl\{
      f \in C(D) \colon
      \frac{\partial^nf}{\partial x^n} \in C(D),
      \frac{\partial^kf}{\partial t^k} \in C(D)
    \Bigr\} \,.
\end{gather}
Similarly, we write $C(\R \times [0,T])$ to denote the continuous
real-valued functions on $\R \times [0,T]$, and
\begin{multline}
  C^{n,k}(\R\times[0,T])
  = \Bigl\{
      f \in C(\R\times[0,T]) \colon \\
      \frac{\partial^nf}{\partial x^n} \in C(\R\times[0,T]),
      \frac{\partial^kf}{\partial t^k} \in C(\R\times[0,T])
    \Bigr\} \,.
\end{multline}
\end{subequations}

It will be convenient to extend the spatial domain of the HHMO-model
to the entire real line by even reflection.  We write out the notation
of weak solutions in this sense, knowing that we can always go back to
the positive half-line by restriction. 

\begin{definition}
\label{weak.sol.def}
A \emph{weak solution} to problem \eqref{e.original} is a pair $(u,p)$
satisfying
\begin{enumerate}[label={\upshape(\roman*)}]
  \item $u$ and $p$ are symmetric in space, i.e.\ $u(x,t)=u(-x,t)$ and
  $p(x,t)=p(-x,t)$ for all $x \in \R$ and $t\ge0$,
  \item\label{weak.ii} $u-\psi\in C^{1,0}(\R\times[0,T])\cap L^{\infty}(\R\times[0,T])$ for
  every $T>0$,
  \item\label{weak.iii} $p$ is measurable and satisfies
  \eqref{e.hhmo-p-weak-alternative},
\item\label{weak.3.5} $p(x,t)$ is non-decreasing in time $t$ for every
$x \in \R$,
\item\label{weak.iv} the relation
\begin{equation}
\label{weak.sol.def.eq}
  \int_0^T\int_\R\varphi_t \, (u-\psi) \, \d y \, \d s
  = \int_0^T \int_\R
    \bigl(
      \varphi_x \, (u-\psi)_x + p \, u \, \varphi
    \bigr) \, \d y \, \d s
\end{equation}
holds for every $\varphi\in C^{1,1}(\R\times[0,T])$ that vanishes for
large values of $|x|$ and for time $t=T$.
\end{enumerate}
\end{definition}

\begin{remark}
The regularity class for weak solutions we require here is less strict
than the regularity class assumed by Hilhorst \emph{et al.}\
\cite[Equation~12]{HilhorstHM:2009:MathematicalSO}, who consider
solutions of class
\begin{equation}
   u-\psi \in C^{1+\ell,\frac{1+\ell}2}(\R\times[0,T])
  \cap H^1_{\loc}(\R\times[0,T])  \,.
\end{equation}
for every $\ell\in(0,1)$, where $C^{\alpha,\beta}$ denote the usual
H\"older spaces; see, e.g., \cite{LadyzhenskajaSU:1968:LinearQP}.
They prove existence of a weak solution in this stronger sense.
Clearly, every weak solution in their setting are solutions to our
problem.  The question of uniqueness is open for both formulations,
but partial results are available
\cite{Darbenas:2018:PhDThesis, DarbenasHO:2020:ConditionalUS}.
\end{remark}

\begin{remark}
\label{r.monotonicity}
The monotonicity condition \ref{weak.3.5} is not included in the
definition of weak solutions by Hilhorst \emph{et al.}\
\cite{HilhorstHM:2009:MathematicalSO}.  Their construction, however,
always preserves monotonicity so that existence of solutions
satisfying this condition is guaranteed.  In the following, it is
convenient to assume monotonicity.  We note that, due to condition
\eqref{e.hhmo-p-weak-alternative}, monotonicity only ever becomes an
issue when $u$ grazes, but does not exceed the precipitation threshold
on sets of positive measure in space-time.  We do not know if such
highly degenerate solutions exist, but the results in
\cite{DarbenasO:2021:BreakdownLP} suggest that this might be the case.
We also remark that the definition of a \emph{completed relay} by
Visintin \cite{Visintin:1986:EvolutionPH,Visintin:1994:DifferentialMH}
includes the requirement of monotonicity.
\end{remark}

To proceed, we introduce some more notation. When $u^*<\Psi(\alpha)$,
we write $\alpha^*$ to denote the unique solution to
\begin{equation}
  \label{alpha.star}
  \Psi(\alpha^*)=u^* \,,
\end{equation}
where $\Psi$ is the precipitation-less solution given by equation
\eqref{self.similar.p0}, and we set
\begin{equation}
\label{d.star}
 D^* = \{(x,t) \colon 0<\alpha^*\sqrt t<x\} \,.
\end{equation}
Further, we abbreviate $[f-g](y,s)=f(y,s)-g(y,s)$ and
$[fg](y,s)=f(y,s) \, g(y,s)$.

In the following, we prove a number of properties which are implied by
the notion of weak solution.  In these proofs, as well as further in
this paper, we rely on the fact that we can read
\eqref{weak.sol.def.eq} as the weak formulation of a \emph{linear}
heat equation of the form
\begin{equation}
  w_t - w_{xx} = g(x,t) \,,
  \label{e.w-strong}
\end{equation}
for a given bounded integrable right-hand function $g$.  We shall
write the equations in their classical form \eqref{e.w-strong} where
convenient with the understanding that they are satisfied in the sense
of \eqref{weak.sol.def.eq}.  Further, in the functional setting of
Definition~\ref{weak.sol.def}, the solution is regular enough such
that it is unique for fixed $g$, the Duhamel formula holds true, and,
consequently, the subsolution resp.\ supersolution principle is
applicable.  For a detailed verification of these statement from first
principles, see, e.g., \cite[Appendix~B]{Darbenas:2018:PhDThesis}.

\begin{lemma}
\label{u.psi}
Any weak solution $(u,p)$ of \eqref{e.original} satisfies
$[u-\psi](x,0)=0$, $0<u\le\psi$ for $t>0$, and $p=0$ on $D^*$.
\end{lemma}

\begin{proof}
The inequality $u\le\psi$ is a direct consequence of the subsolution
principle.  Hence, $u\le\psi<u^*$ on $D^*$, so $p=0$ on $D^*$.  Now
consider the weak solution to
\begin{subequations}
\begin{gather}
  u^\ell_t = u^\ell_{xx} +
        \frac{\alpha \beta}{2 \sqrt t} \, \delta (x - \alpha \sqrt{t})
        - u^\ell \,,\\
  u^\ell_x(0,t) = 0 \quad \text{for } t > 0 \,, \\
  u^\ell(x,0) = 0 \quad \text{for } x>0 \,,
\end{gather}
\end{subequations}
which transforms into
\begin{equation}
  (\e^t \, u^\ell)_t
  = (\e^t \, u^\ell)_{xx} + \e^t \, \frac{\alpha \beta}{2 \sqrt t} \,
    \delta (x - \alpha \sqrt{t}) \,.
\end{equation}
As the distribution on the right hand side is positive, the Duhamel
principle implies that $\e^t \, u^\ell$ is positive for $t>0$, and so
is $u^\ell$.  Due to the subsolution principle, we find
$u\ge u^\ell > 0$ for $t>0$.  Finally, since
$\lim_{t\to\infty} \psi(x,t)=0$ for $x>0$ fixed, this implies
$[u-\psi](x,0)=0$.
\end{proof}

\begin{lemma}
\label{p.dependent.u}
The precipitation function $p$ is essentially determined by the
concentration field $u$, i.e., if $(u,p_1)$ and $(u,p_2)$ are weak
solutions to \eqref{e.original} on $\R\times[0,T]$, then $p_1= p_2$
almost everywhere on $\R\times[0,T]$.
\end{lemma}

\begin{proof}
Taking the difference of \eqref{weak.sol.def.eq} with $p=p_1$ and
$p=p_2$, we find
\begin{equation}
	\int_0^t\int_\R (p_1-p_2)\, u \, \varphi\,\d x\, \d t = 0
\end{equation}
for every $\varphi\in C^{1,1}(\R\times[0,T])$ that vanishes for large
values of $|x|$ and time $t=T$.  As such functions are dense in
$L^2(\R\times[0,T])$, we conclude $(p_1-p_2) \, u=0$ a.e.\ in
$\R\times[0,T]$.  Moreover, $u>0$ for $t>0$, so that $p_1=p_2$ a.e.\
in $\R\times[0,T]$.
\end{proof}

\begin{theorem}[Weak solutions with subcritical precipitation threshold]
\label{weak.sol.ustar.threshold.subcritical}
When $u^*>\Psi(\alpha)$, then $(\psi,0)$ is the unique weak
solution of \eqref{e.original}.
\end{theorem}

\begin{proof}
We know that $u\le\psi$ from Lemma~\ref{u.psi}. Therefore, the
threshold $u^*$ will be never reached. So $p=0$ and, due to the
uniqueness of weak solutions for linear parabolic equations, $u=\psi$.
\end{proof}

The following result shows that, in general, we cannot expect
uniqueness of weak solutions:  When the precipitation threshold is
marginal, the concentration can remain at the threshold for large
regions of space-time.  Within such regions, spontaneous onset of
precipitation is possible on arbitrary subsets, thus a large number of
nontrivial weak solutions exists.
The precise result is the following.

\begin{theorem}[Weak solutions with marginal precipitation threshold]
\label{weak.sol.ustar.threshold.critical}
When $u^*=\Psi(\alpha)$, the set of weak solutions to
\eqref{e.original} is equal to the set of pairs $(u,p)$ such that
\begin{enumerate}[label={\upshape(\roman*)}]
\item\label{wsut.i} $p$ is an even measurable function taking values
in $[0,1]$,
\item $p(x,t)$ is non-decreasing in time $t$ for every
$x \in \R$,
\item\label{wsut.ii} there exists $b>0$ such that $p(x,t)=0$ if
$(x,t) \notin U=[-b,b]\times[b^2/\alpha^2, \infty)$,
\item\label{wsut.iii} $(u,p)$ satisfies the weak form of the equation of
motion, i.e., Definition~\textup{\ref{weak.sol.def}}~\ref{weak.iv}
holds true.
\end{enumerate}
\end{theorem}

\begin{proof}
Assume that $(u,p)$ is any pair satisfying
\ref{wsut.i}--\ref{wsut.iii}.  To show that $(u,p)$ is a weak
solution, we need to verify that it is compatible with condition
\eqref{e.hhmo-p-weak-alternative}; all other properties are trivially
satisfied by construction.  Since $u\le\psi\le\Psi(\alpha)=u^*$, it
suffices to prove that $p(x,t)>0$ implies
$\max_{\tau\in[0,t]}u(x,\tau) = u^*$.  We begin by observing that
$u(x,t)=\psi(x,t)$ for all $x\in\R$ if $t\in[0,b^2/\alpha^2]$.  Since,
by construction, $p(x,t)>0$ only for $(x,t)\in U$, this implies
\begin{equation}
  \max_{t\in[0,b^2/\alpha^2]}u(x,t)
  \ge u(x,x^2/\alpha^2)
  = \psi(x,x^2/\alpha^2) = u^* \,.
\end{equation}
In other words, $(u,p)$ is compatible with
\eqref{e.hhmo-p-weak-alternative} on $U$.  For $(x,t) \notin U$,
$p(x,t)=0$ and \eqref{e.hhmo-p-weak-alternative} is trivially
satisfied.  Altogether, this proves that that $(u,p)$ is a weak
solution on the whole domain $\R\times\R_+$.

Vice versa, assume that $(u,p)$ is a weak solution.  If $p=0$ a.e.,
then $u=\psi$ and \ref{wsut.i}--\ref{wsut.iii} are satisfied for any
$b>0$.  Otherwise, define
\begin{gather}
  A(t) = \{ (x,\tau) \colon \tau \leq t \text{ and } p(x,\tau)>0 \} \,, \\
  T = \inf \{ t>0 \colon m(A(t))>0 \}\,,
\end{gather}
where $m$ denotes the two-dimensional Lebesgue measure.  By
definition, $p=0$ a.e.\ on $\R\times[0,T]$ so that $u=\psi$ on
$\R\times[0,T]$.  We also note that
\begin{gather}
  m(\{(x,t) \colon t\in[T,T+\varepsilon] \text{ and } p(x,t)>0\})>0
\end{gather}
for every $\varepsilon>0$ and that $u(x,t)>0$ for all $t>0$.  Then for
every $t>T$, by the Duhamel principle,
\begin{equation}
  u(x,t)
  = \psi(x,t)
    - \int_0^t\int_\R\HK(y,\tau) \, [pu](x-y,t-\tau) \, \d y\,\d \tau
  < \psi(x,t) \le \Psi(\alpha) \,,
  \label{e.inequ1}
\end{equation}
where $\HK$ is the standard heat kernel
\begin{equation}
  \HK(x,t) = \begin{dcases}\frac1{\sqrt{4\pi t}} \,
                 \e^{-\tfrac{x^2}{4t}}&\text{if }t>0 \,,\\
		 0&\text{if }t\le0 \,.
	      \end{dcases}
\end{equation}
We first note that $T>0$.  Indeed, if $T$ were zero, \eqref{e.inequ1}
would imply that $u(x,t)<u^*$ for all $x\neq 0$, so that $p=0$ a.e., a
contradiction.  Moreover, taking $\abs{x}>\alpha\sqrt{T}$,
\begin{equation}
  \max_{t\in[0,T]}u(x,t)
  \le \max_{t\in[0,T]} \psi(x,t)
  = \psi(x,T)
  < \Psi(\alpha) \,.
  \label{e.inequ2}
\end{equation}
Inequalities \eqref{e.inequ1} and \eqref{e.inequ2} imply that
$p(x,t)=0$ so that \ref{wsut.i}--\ref{wsut.iii} are satisfied with
$b=\alpha\sqrt{T}>0$.
\end{proof}

\begin{remark}
\label{r.nonunique}
Theorem~\ref{weak.sol.ustar.threshold.critical} illustrates how
non-uniqueness of weak solutions arises in the case of a marginal
precipitation threshold.  One obvious solution is $u=\psi$ and $p=0$.
Solutions with nonvanishing precipitation can be constructed as
follows.  Fix any $b>0$ and take any even measurable function $p^*$
taking values in $[0,1]$ with $\supp p^* \subset [-b,b]$.  Set
$p(x,t) = p^*(x) \, H(t-b^2/\alpha^2)$.  Then $p$ satisfies
\ref{wsut.i}--\ref{wsut.ii}.  On the time interval $[0,b^2/\alpha^2]$,
$u=\psi$ satisfies the weak form.  For $t>b^2/\alpha^2$, determine $u$
as the weak solution to the linear parabolic equation
\eqref{e.original.a} with the given function $p$.  Then, by
construction, $(u,p)$ is a weak solution in the sense of
Definition~\ref{weak.sol.def}.
\end{remark}

\begin{remark}
\label{r.prec-condition}
Theorem~\ref{weak.sol.ustar.threshold.critical} admits more weak
solutions than those described in Remark~\ref{r.nonunique}.  We note
that, in particular, the precipitation condition
\eqref{e.hhmo-p-weak-alternative} allows ``spontaneous precipitation''
even when the maximum concentration has fallen below the precipitation
threshold everywhere provided the concentration has been at the
threshold at earlier times.  This behavior should be considered
unphysical and is discarded, for the purposes of this paper, by
imposing condition (P).
\end{remark}

The following result shows that the concentration $u$ is uniformly
Lipschitz in $x$-$t$ coordinates.  It does not imply a uniform
Lipschitz estimate with respect to the spatial similarity coordinate
$\eta$; due to the change of coordinates, the constant will grow
linearly in $t$.  However, the conjectured asymptotics of the
precipitation function implies uniformity in similarity variables. We
will not use this result in the remainder of the paper, but state it
here as the best estimate which we were able to obtain by direct
estimation in the Duhamel formula or using energy methods.

\begin{lemma} \label{l.Lipschitz}
Let $(u,p)$ be a weak solution to \eqref{e.original}.  Then, for any
$T>0$, $u$ is uniformly Lipschitz continuous on $\R \times [T,\infty)$.
\end{lemma}
\begin{proof}
Let $w = \psi - u$.  A weak solution must satisfy the Duhamel formula
(see, e.g., \cite[Appendix~B]{Darbenas:2018:PhDThesis}),
so
\begin{align}
  w(x_2,t) - w(x_1,t)
  & = \int_0^t \int_\R
        \bigl( \HK(x_2-y, t-\tau) - \HK(x_1-y, t-\tau) \bigr) \,
        [pu](y,\tau) \, \d y \, \d \tau
      \notag \\
  & \equiv W_{[0,t-\delta]} + W_{[t-\delta,t]} \,,
\end{align}
where we split the domain of time-integration into two subintervals
and write $W_I$ to denote the contribution from subinterval $I$.  In
the following, we suppose that $x_1<x_2$ and choose
$\delta = \min \{ t, \tfrac14 \}$.

On the subinterval $[0, t-\delta]$, if not empty, we apply the
fundamental theorem of calculus, so that
\begin{equation}
  \lvert W_{[0,t-\delta]} \rvert
  = \int_0^{t-\delta} \int_\R \int_{x_1}^{x_2}
        \HK_x(\xi-y, t-\tau) \, \d \xi \, [pu](y,\tau) \, \d y \, \d \tau \,.
  \label{e.diff-int1}
\end{equation}
Now note that
\begin{align}
  \lvert \HK_x(\xi-y, t-\tau) \rvert
  & = \frac1{4 \sqrt\pi} \,
      \frac{\lvert \xi-y \rvert}{(t-\tau)^{3/2}} \,
      \e^{-\tfrac{(\xi-y)^2}{4(t-\tau)}}
      \notag \\
  & = \frac{\lvert \xi-y \rvert \, \sqrt{t-\tau+\delta}}%
           {2 \, (t-\tau)^{3/2}} \,
      \e^{-\tfrac{(\xi-y)^2 \, \delta}{4(t-\tau)(t-\tau+\delta)}} \,
      \frac1{\sqrt{4 \pi (t-\tau+\delta)}} \,
      \e^{-\tfrac{(\xi-y)^2}{4(t-\tau+\delta)}}
      \notag \\
  & = \frac1{\sqrt\delta} \,
      \biggl( 1 + \frac{\delta}{t-\tau} \biggr) \,
      \zeta \, \e^{-\zeta^2} \,
      K(\xi-y,t-\tau+\delta)
      \notag \\
  & \leq c(\delta) \, K(\xi-y,t-\tau+\delta) \,,
  \label{e.kx-estimate}
\end{align}
where we have defined
\begin{equation}
  \zeta = \lvert \xi - y \rvert \,
          \frac{\sqrt\delta}{2 \, \sqrt{t-\tau} \, \sqrt{t-\tau+\delta}}
\end{equation}
and, to obtain the final inequality in \eqref{e.kx-estimate}, note
that $\zeta \, \e^{-\zeta^2}$ is bounded and $t-\tau \geq \delta$.
Changing the order of integration in \eqref{e.diff-int1}, taking
absolute values, and inserting estimate \eqref{e.kx-estimate}, we
obtain
\begin{align}
  \lvert W_{[0,t-\delta]} \rvert
  & \leq c(\delta) \, \int_{x_1}^{x_2} \int_0^{t-\delta} \int_{\R}
           K(\xi-y,t-\tau+\delta) \, [pu](y,\tau) \, \d y \, \d \tau \, \d \xi
         \notag \\
  & \leq c(\delta) \, \int_{x_1}^{x_2} \int_0^{t+\delta} \int_{\R}
           K(\xi-y,t+\delta-\tau) \, [pu](y,\tau) \, \d y \, \d \tau \, \d \xi
         \notag \\
  & \leq c(\delta) \, \lvert x_2 - x_1 \rvert \,
         \sup_{\xi \in \R} \, \abs{w(\xi,t+\delta)} \,.
\end{align}
Since $w$ is bounded, we have obtained a uniform-in-time Lipschitz
estimate for $w$ on the first subinterval.

On the subinterval $[t-\delta,t]$, we use the boundedness of $pu$, so
that we can take out this contribution in the space-time $L^\infty$
norm,
\begin{equation}
  \lvert W_{[t-\delta,t]} \rvert
  \leq \int_{t-\delta}^t \int_\R
      \lvert \HK(x_2-y, t-\tau) - \HK(x_1-y, t-\tau) \rvert \,
      \d y \, \d \tau \, \lVert pu \rVert_{L^\infty} \,.
\end{equation}
Setting $r = (x_2-x_1)/2$ and changing variables $t-\tau \mapsto \tau$, we
obtain
\begin{align}
  \int_{t-\delta}^t \int_\R
      & \lvert \HK(x_2-y, t-\tau) - \HK(x_1-y, t-\tau) \rvert \, \d y \, \d \tau
      \notag \\
  & = \int_0^\delta
      \biggl(
        \erfc \Bigl(-\frac{r}{2 \sqrt \tau} \Bigr)
        - \erfc \Bigl(\frac{r}{2 \sqrt \tau} \Bigr)
      \biggr) \, \d \tau
      \notag \\
  & \leq \int_0^{1/4}
      \biggl(
        \erfc \Bigl(-\frac{r}{2 \sqrt \tau} \Bigr)
        - \erfc \Bigl(\frac{r}{2 \sqrt \tau} \Bigr)
      \biggr) \, \d \tau
      \notag \\
  & = \frac{r}{\sqrt \pi} \, \e^{-r^2}
      + \frac12 \, \erf (r) - r^2 \, (1- \erf (r))
      \notag \\
  & \leq c \, \lvert x_2 - x_1 \rvert \,,
  \label{e.estimate-int2}
\end{align}
where the last inequality is based on the observation that $\erf (r)$
is a smooth odd concave function and that $r \, (1-\erf (r))$ is
bounded.  This proves a uniform-in-time Lipschitz estimate for $w$ on
the second subinterval as well.  Since $\psi$ is uniformly Lipschitz
on $\R \times [T,\infty)$ by direct inspection, $u=\psi-w$ is
uniformly Lipschitz on the same domain.
\end{proof}


\begin{remark}
We note that the heat equation with arbitrary $L^\infty$ right-hand
side is not necessarily uniformly Lipschitz.  This can be seen by
observing that if we carry out the integration in
\eqref{e.estimate-int2} with arbitrary $\delta$, the constant $c$ will
be proportional to $\sqrt \delta$.  Thus, choosing $\delta=t$, thereby
eschewing the separate estimate for the first subinterval, we obtain a
Lipschitz constant which grows like $\sqrt t$.  Without recourse to the
particular features of the HHMO-model, this result is sharp, as can be
seen by taking the standard step function as right-hand function for
the heat equation.
\end{remark}

\section{Self-similar solution for self-similar precipitation}
\label{s.self-similar}

The computation of Section~\ref{s.without} can be extended to the case
when the precipitation term in $\eta$-$s$ coordinates does not have
any explicit dependence on $s$.  To do so, it is necessary that
precipitation is a function of the similarity variable $\eta$ only,
which requires that $q(\eta,s)=p(s \eta) = \gamma/(s\eta)^2$ for some
constant $\gamma>0$ which we treat as an unknown.  This means that we
disregard \eqref{e.hhmo-p-weak-alternative} which defines the
precipitation function in the original HHMO-model.  We also disregard
the requirement that $p \in [0,1]$ in the definition of the
generalized precipitation function \eqref{e.hhmo-p-weak}.  With these
provisions, the coefficients of the right hand side of \eqref{e.v} do
not depend on $s$. Therefore, as we shall show in the following,
steady states which we call \emph{self-similar solutions} indeed
exist, and we establish sufficient and necessary conditions for their
existence.

As before, we seek a stationary solution for \eqref{e.v}, which now
reduces to
\begin{subequations}
  \label{e.v-selfsimilar}
\begin{gather}
  \Phi'' + \frac\eta2 \, \Phi' + \frac{\alpha\beta}2 \, \delta(\eta-\alpha)
  - \frac\gamma{\eta^2} \, H(\alpha-\eta) \, \Phi = 0 \,,
  \label{e.ode2} \\
  \Phi(\eta) \to 0 \quad \text{as } \eta \to \infty \,,
  \label{e.v-selfsimilar-c} \\
  \Phi(\alpha) = u^* \,,
  \label{e.v-selfsimilar-d} \\
  \Phi'(0) = 0 \,. \label{e.v-selfsimilar-b}
\end{gather}
\end{subequations}
The additional internal boundary condition \eqref{e.v-selfsimilar-d}
models the observation that the HHMO-model drives the solution to the
critical value $u^*$ along the line $\eta=\alpha$.  As we will show
below, subject to a certain solvability condition, there will be a
unique pair $(\Phi,\gamma)$ solving this system.

We interpret the derivatives in \eqref{e.ode2} in the sense of
distributions, so that
\begin{gather}
  \Phi'(\eta) = \frac{\d \Phi}{\d \eta} + [\Phi(\alpha)] \, \delta_\alpha
  \label{e.vprime}
\intertext{and}
  \Phi''(\eta) = \frac{\d^2 \Phi}{\d \eta^2}
    + [\Phi'(\alpha)] \, \delta_\alpha + [\Phi(\alpha)] \, \delta'_\alpha \,,
  \label{e.vpprime}
\end{gather}
where $[\Phi(\alpha)] = \Phi(\alpha+)-\Phi(\alpha-)$ and $\d/\d \eta$
denotes the classical derivative where the function is smooth, i.e.,
on $(0,\alpha)$ and $(\alpha,\infty)$, and takes any finite value at
$\eta=\alpha$ where the classical derivative may not exist.  Inserting
\eqref{e.vprime} and \eqref{e.vpprime} into \eqref{e.ode2}, we obtain
\begin{multline}
  \frac{\d^2 \Phi}{\d \eta^2} + \frac\eta2 \, \frac{\d \Phi}{\d \eta}
  - \frac\gamma{\eta^2} \, H(\alpha-\eta) \, \Phi
  + \biggl(
      \frac{\alpha\beta}2 + \frac\eta2 \, [\Phi(\alpha)]
      + [\Phi'(\alpha)]
    \biggr) \, \delta_\alpha
  + [\Phi(\alpha)] \, \delta'_\alpha = 0 \,.
\end{multline}
Going from the most singular to the least singular term, we conclude
first that $[\Phi(\alpha)]=0$, i.e., that $\Phi$ is continuous across
the non-smooth point at $\eta=\alpha$.  Second, we obtain a jump
condition for the first derivative, namely
\begin{equation}
  [\Phi'(\alpha)] = - \frac{\alpha \beta}2 \,.
  \label{e.jump-condition}
\end{equation}

On the interval $(\alpha,\infty)$, we need to solve
\begin{subequations}
  \label{e.vupper}
\begin{gather}
  \Phi_\r'' + \frac\eta2 \, \Phi_\r' = 0 \,, \label{e.ode4} \\
  \Phi_\r(\eta) \to 0 \quad \text{as } \eta \to \infty \,.
\end{gather}
\end{subequations}
As in Section~\ref{s.without}, the solution to \eqref{e.vupper} is of
the form
\begin{subequations}
  \label{e.rightsolution}
\begin{equation}
  \Phi_\r(\eta) = C_2 \, \erfc(\eta/2)
\end{equation}
where, due to the internal boundary condition $\Phi(\alpha)=u^*$,
\begin{equation}
  C_2 = \frac{u^*}{\erfc (\frac\alpha2)} \,.
  \label{e.c2}
\end{equation}
\end{subequations}
Its derivative is given by
\begin{equation}
  \Phi_\r'(\eta) = -C_2 \, \frac{\exp(-\eta^2/4)}{\sqrt\pi} \,.
  \label{e.vupper-deriv}
\end{equation}

Similarly, on the interval $(0,\alpha)$, we need to solve
\begin{subequations}
\begin{gather}
  \Phi_\l'' + \frac\eta2 \, \Phi_\l' - \frac\gamma{\eta^2} \, \Phi_\l = 0 \,,
  \label{e.ode3} \\
  \Phi_\l'(0) = 0 \,.
\end{gather}
\end{subequations}
Equation \eqref{e.ode3} is a particular instance of the general
confluent equation \cite[Equation
13.1.35]{AbramowitzS:1972:HandbookMF}, whose solution is readily
expressed in terms of Kummer's confluent hypergeometric function $M$,
also referred to as the confluent hypergeometric function of the first
kind ${}_1\!F_1$.  The two linearly independent solutions are of the
form
\begin{subequations}
  \label{e.leftsolution}
\begin{equation}
  \Phi_\l(\eta) = C_1 \, \eta^\kappa \,
  M \Bigl(\frac\kappa2,\kappa+\frac12, -\frac{\eta^2}4 \Bigr)
  \label{e.vlowersol}
\end{equation}
where $\kappa (\kappa-1) = \gamma$ and, due to the internal boundary
condition $\Phi(\alpha)=u^*$,
\begin{equation}
  C_1 = \frac{u^*}{\alpha^\kappa
  M \bigl(\frac\kappa2,\kappa+\frac12,-\frac{\alpha^2}4 \bigr)} \,.
  \label{e.c1}
\end{equation}
\end{subequations}
Solving for $\kappa$, we find that of the two roots
\begin{equation}
  \kappa_{1,2}=\frac{1 \pm \sqrt{4\gamma+1}}2 \,,
  \label{e.kappa}
\end{equation}
only the larger one is positive, corresponding to regular behavior of
the solution \eqref{e.vlowersol} at the origin.  When
$\kappa_2+\frac12$ is not a negative integer, \eqref{e.leftsolution}
provides a second linearly independent solution with $\kappa=\kappa_2$
which we discard as it has a pole at $\eta=0$.  When
$\kappa_2+\frac12$ is a negative integer, Kummer's function is not
defined, so that we use the method of method of reduction of order,
see \cite[Section 3.4]{Teschl:2012:OrdinaryDE}, to obtain a second
linearly independent solution.  To do so, we assume that
$\Phi(\eta)=e(\eta) \, \Phi_\l(\eta)$ and obtain an equation for $e$,
\begin{equation}
  e''+ \Bigl(2 \, \frac{\Phi_\l'}{\Phi_\l} + \frac\eta2 \Bigr) \, e'
  = 0
\end{equation}
on $(0,\alpha]$.  Integrating, we obtain
\begin{subequations}
\begin{gather}
  e'(\eta) = C_e \, \Phi_\l^{-2}(\eta) \, \e^{-\frac{\eta^2}4} \,, \\
  e(\eta) = -C_e \int^{\alpha}_\eta \Phi_\l^{-2}(\zeta) \,
  \e^{-\frac{\zeta^2}4} \, \d\zeta + C_e^* \,,
\end{gather}
\end{subequations}
again on $(0,\alpha]$.  Hence, the general solution to \eqref{e.ode2}
on $(0,\alpha]$ is
\begin{equation}
  \Phi(\eta)
  = -C_e\,\Phi_\l(\eta)\,\int^{\alpha}_\eta \Phi_\l^{-2}(\zeta) \,
    \e^{-\frac{\zeta^2}4} \, \d\zeta + C_e^* \, \Phi_\l(\eta) \,.
  \label{e.vlowersol.pole}	
\end{equation}
To obtain a second linearly independent solution, it suffices to take
$C_e=1$ and $C_e^*=0$.  We proceed to show that the first term on the
right has again a pole at $\eta$.  Identity \cite[Equation
13.1.27]{AbramowitzS:1972:HandbookMF} reads
\begin{equation}
   M\Bigl(\frac{\kappa_1}2,\kappa_1+\frac12, -\frac{\eta^2}4 \Bigr)
   = \e^{-\frac{\eta^2}4} \,
     M\Bigl(\frac{\kappa_1}2+\frac12,\kappa_1+\frac12, \frac{\eta^2}4 \Bigr)
   > 0 \,.
\end{equation}
Due to \eqref{e.vlowersol}, we can find a positive constant $C$ such
that
\begin{equation}
  e(\eta)
  \le -C \int^{\alpha}_\eta \zeta^{-2\kappa_1} \, \d\zeta
  = - \frac{C}{2\kappa_1-1} \,
    \bigl( \eta^{-2\kappa_1+1}-\alpha^{-2\kappa_1+1} \bigr) \,.
\end{equation}
Therefore,
\begin{equation}
  \Phi(\eta)
  \le -\frac{C\,C_1}{2\kappa_1-1} \,
      \bigl( \eta^{-\kappa_1+1}-\alpha^{-2\kappa_1+1} \,
        \eta^{\kappa_1} \bigr) \,
      M \Bigl(\frac{\kappa_1}2,\kappa_1+\frac12, -\frac{\eta^2}4
        \Bigr) \,.
\end{equation}
Thus, the second linearly independent solution again has a pole at
$\eta=0$.  Therefore, we consider $\kappa=\kappa_1$ onward only.

Using the properties of Kummer's function
\cite[Section~13.4]{AbramowitzS:1972:HandbookMF}, the derivative of
\eqref{e.vlowersol} is readily computed as
\begin{equation}
  \Phi_\l'(\eta) = C_1 \, \kappa \, \eta^{\kappa-1} \,
  M \Bigl(\frac\kappa2+1,\kappa+\frac12, -\frac{\eta^2}4 \Bigr) \,.
  \label{e.vlowersold-deriv}
\end{equation}

Finally, we use the jump condition \eqref{e.jump-condition} to
determine the constant $\gamma$.  Plugging the left-hand and
right-hand solution into \eqref{e.jump-condition}, we find
\begin{equation}
  \label{eq.kappa}
  u^* \,
  \frac{\kappa \,
    M \bigl( \frac\kappa2+1,\kappa+\frac12, -\frac{\alpha^2}4 \bigr)}%
   {\alpha \,
    M \bigl( \frac\kappa2,\kappa+\frac12,-\frac{\alpha^2}4 \bigr)}
  + u^* \, \frac{\exp \bigl(-\frac{\alpha^2}4 \bigr)}%
                {\sqrt\pi \, \erfc (\frac\alpha2 )}
  = \frac{\alpha\beta}2 \,.
\end{equation}

\begin{figure}
\centering
\includegraphics{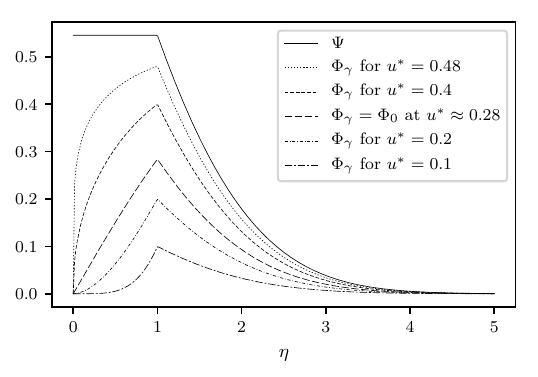}
\caption{Plot of $\Psi$ and of the family of profiles $\Phi_\gamma$
for different precipitation thresholds $u^*$.  The profiles in between
$\Psi$ and $\Phi_0$ correspond to the transitional regime where
$\gamma$ is negative, hence they fall outside
of the class of self-similar solutions described by
Theorem~\ref{t.selfsimilar}.  Solutions to the HHMO-model in the
transitional regime always converge to $\Phi_0$, not to $\Phi_\gamma$
with $\gamma<0$.}
\label{f.profiles}
\end{figure}

To proceed, we set
\begin{equation}
  \label{eq.kappa.2}
  u^*_\gamma =
  \left(\frac{\kappa \,
    M \bigl( \frac\kappa2+1,\kappa+\frac12, -\frac{\alpha^2}4 \bigr)}%
   {\alpha \,
    M \bigl( \frac\kappa2,\kappa+\frac12,-\frac{\alpha^2}4 \bigr)}
  + \, \frac{\exp \bigl(-\frac{\alpha^2}4 \bigr)}%
                {\sqrt\pi \, \erfc (\frac\alpha2 )}\right)^{-1}
  \frac{\alpha\beta}2
\end{equation}
and join the right-hand solution \eqref{e.rightsolution} and left-hand
solution \eqref{e.leftsolution} to define a family of functions,
parameterized by $\gamma$, by
\begin{gather}
  \Phi_\gamma(\eta)
  = \begin{dcases}
      \frac{u^*_\gamma\,\eta^\kappa \,
            M \bigl(\frac\kappa2,\kappa+\frac12, -\frac{\eta^2}4 \bigr)}
           {\alpha^\kappa \, M \bigl(\frac\kappa2,\kappa+\frac12,
            -\frac{\alpha^2}4 \bigr)} & \text{ if }\eta<\alpha \,, \\
      \frac{u^*_\gamma}{\erfc (\frac\alpha2)} \,
      \erfc \Bigl( \frac\eta2 \Bigr)
      & \text{ if }\eta\ge\alpha \,.
    \end{dcases}
  \label{e.phi.gamma}
\end{gather}
For future reference, we note that in $x$-$t$ coordinates, this
function takes the form
\begin{equation}
\label{e.phi.gamma.x-t}
   \phi_\gamma(x,t)=\Phi_\gamma(x/{\sqrt t}) \,.
\end{equation}
At this point, we know that each $\Phi_\gamma(\eta)$ satisfies the
differential equation \eqref{e.ode2} and the decay condition
\eqref{e.v-selfsimilar-c}.  However, $\Phi_\gamma$ does not
necessarily satisfy the internal boundary condition
\eqref{e.v-selfsimilar-d}, equivalent to the matching condition
\eqref{eq.kappa} which can now be expressed as $u^*_\gamma = u^*$, nor
does it necessarily satisfy the Neumann boundary condition
\eqref{e.v-selfsimilar-b}, which requires $\gamma>0$ or, equivalently,
$\kappa>1$.  The following theorem states a necessary and sufficient
condition such that \eqref{eq.kappa} can be solved for $\kappa>1$ or,
equivalently, $u^*_\gamma = u^*$ can be solved for $\gamma>0$.  When
this is the case, the resulting matched solution solves the entire
system \eqref{e.v-selfsimilar}.

\begin{theorem}
\label{t.selfsimilar}
Let $\alpha$, $\beta$, and $u^*$ be positive.  Then the matching
condition $u^*_\gamma = u^*$ has a unique solution satisfying
$\gamma>0$ if and only if $u^*<u_0^*$.  If this is the case, the
unique solution to \eqref{e.v-selfsimilar} is given by
\eqref{e.phi.gamma} with this particular value of $\gamma$.
\end{theorem}

\begin{remark} 
We recall that for a \emph{subcritical} precipitation threshold where
$u^* > \Psi(\alpha)$, no precipitation can occur and $\Psi$, defined
in \eqref{self.similar.p0}, provides a self-similar solution without
precipitation.  The \emph{marginal} case $u^* = \Psi(\alpha)$ is
discussed in Theorem~\ref{weak.sol.ustar.threshold.critical}.  In the
\emph{transitional} regime $u_0^* \leq u^* < \Psi(\alpha)$, there is
some $\gamma \leq 0$ so that \eqref{e.phi.gamma} still solves
(\ref{e.v-selfsimilar}a--c); however, $\gamma <0$ is nonphysical and
the Neumann condition \eqref{e.v-selfsimilar-b} cannot be satisfied in
this regime.  For future reference, we call the limiting case
$u^* = u_0^*$ the \emph{critical} precipitation threshold.  In this
case, \eqref{e.phi.gamma} takes the form
\begin{gather}
  \Phi_0(\eta)
  = \begin{dcases}
      \frac{u^*_0}{\erf (\frac\alpha2)} \, \erf (\frac\eta2)
      & \text{ if }\eta<\alpha \,, \\
      \frac{u^*_0}{\erfc (\frac\alpha2)} \, \erfc (\frac\eta2)
      & \text{ if } \eta \ge \alpha \,.
    \end{dcases}
  \label{e.phi0}
\end{gather}
As discussed, this is not a solution, but will emerge as the universal
asymptotic profile for solutions in the transitional regime.  Finally,
the \emph{supercritical} regime $u^*<u_0^*$ is the regime where
Theorem~\ref{t.selfsimilar} provides a self-similar solution to the
HHMO-model with self-similar precipitation function.  The profiles for
the different cases are summarized in Figure~\ref{f.profiles}.
\end{remark}

\begin{proof}[Proof of Theorem~\ref{t.selfsimilar}]
The form of the solution is determined by the preceding construction.
It remains to show that when $u^*<u_0^*$, the derivative matching
condition \eqref{eq.kappa} has a unique solution $\kappa>1$.  Let us
consider the left-hand solution \eqref{e.leftsolution} as a function
of $\eta$ and $\kappa$, which we denote by $v(\eta,\kappa)$, so that
the leftmost term in \eqref{eq.kappa} is $v_\eta(\alpha,\kappa)$.

We begin by noting that
\begin{equation}
  v_\eta(\alpha,1) = u^* \,
  \frac{
    M \bigl( \frac32 ,\frac32, -\frac{\alpha^2}4 \bigr)}%
   {\alpha \,
    M \bigl( \frac12, \frac32, -\frac{\alpha^2}4 \bigr)} \,.
\end{equation}
Moreover,
\begin{equation}
  \lim_{\kappa\to\infty}
    M \bigl(\tfrac\kappa2+1,\kappa+\tfrac12, -\tfrac{\alpha^2}4 \bigr)
  = \lim_{\kappa\to\infty}
    M \bigl( \tfrac\kappa2,\kappa+\tfrac12, -\tfrac{\alpha^2}4 \bigr)
  = \exp \bigl(-\tfrac{\alpha^2}8 \bigr) \,,
\end{equation}
as is easily proved by using the dominated convergence theorem on the
power series representation of Kummer's function.  Consequently,
$v_\eta(\alpha,\kappa)$ grows without bound as $\kappa \to \infty$.
Solvability under the condition that $u^*<u_0^*$ is then a simple
consequence of the intermediate value theorem.

To prove uniqueness, we show that $v_\eta(\alpha,\kappa)$ is strictly
monotonic in $\kappa$.  For fixed $\kappa_2>\kappa_1$, we define
\begin{equation}
  V(\eta) = v(\eta,\kappa_2) - v(\eta,\kappa_1) \,.
  \label{def.V}
\end{equation}

First, $v(\eta,\kappa_1)$ and $v(\eta,\kappa_2)$ satisfy the
differential equation \eqref{e.ode3} with respective constants
$\gamma_1<\gamma_2$.  Thus,
\begin{equation}
  V''(\eta) + \frac\eta2 \, V'(\eta)
  = \frac{\gamma_2}{\eta^2} \, V(\eta)
    + \frac{\gamma_2-\gamma_1}{\eta^2} \, v(\eta,\kappa_1) \,.
  \label{e.Vpp}
\end{equation}

We note that $V(0)=V(\alpha)=0$. Assume that $V$ attains a local
non-negative maximum at $\eta_0\in(0,\alpha)$. Then $V(\eta_0)\ge0$,
$V'(\eta_0)=0$, and $V''(\eta_0)\le0$.  This contradicts \eqref{e.Vpp}
as the left hand side is non-positive and the right hand side is
positive.  We conclude that $V$ is negative in the interior of
$[0,\alpha]$.

In particular, this means that $V'(\alpha) \geq 0$.  The proof is
complete if we show that this inequality is strict.  To proceed,
assume the contrary, i.e., that $V'(\alpha)=0$.  However, inserting
$V'(\alpha)=V(\alpha)=0$ into \eqref{e.Vpp}, we see that there must
exist a small left neighborhood of $\alpha$, $(\alpha_0,\alpha)$ say,
on which $V''$ is positive.  This implies that $V'$ is negative and
$V$ is positive on $(\alpha_0,\alpha)$, which is a contradiction.
\end{proof}

\begin{figure}
\centering
\includegraphics{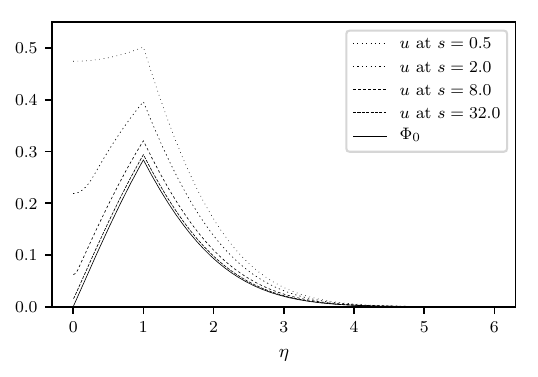}
\caption{Plot of the function $u$ for $\alpha=1.0$, $\beta=1.0$ in the
transitional regime with $u^*=0.49$ for different times $s$, together
with the conjectured limit profile $\Phi_0$.}
\label{f.1}
\end{figure}

\begin{figure}
\centering
\includegraphics{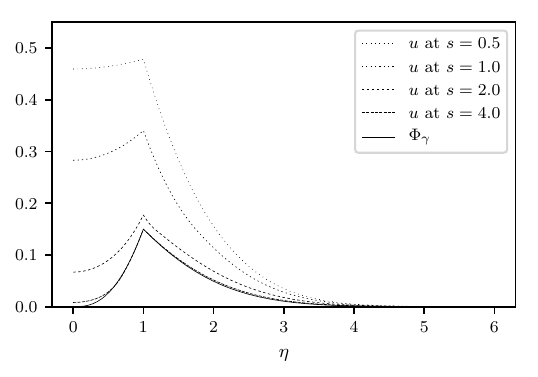}
\caption{Plot of the function $u$ for $\alpha=1.0$, $\beta=1.0$ in the
supercritical regime with $u^*=0.15$ for different times $s$, together
with the conjectured limit profile $\Phi_\gamma$.}
\label{f.2}
\end{figure}

\begin{figure}
\centering
\includegraphics{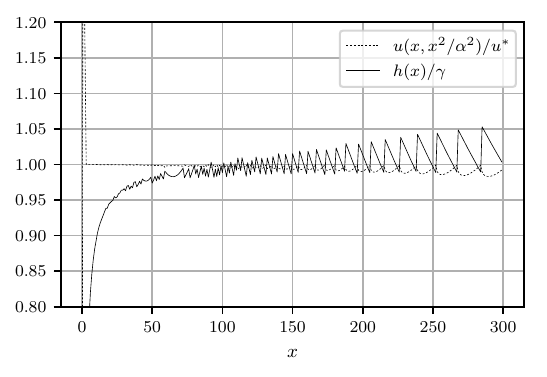}
\caption{Longer-time diagnostics in the supercritical regime.  Shown
are two quantities on the line $\eta=\alpha$ relative to their
conjectured limits for the simulation shown in Figure~\ref{f.2}.  The
growing oscillations are an effect of the finite constant grid size,
see text.}
\label{f.3}
\end{figure}

\section{Numerical results}
\label{s.numerics}

In the following, we present numerical evidence which suggests that
the profiles $\Phi_\gamma$ derived in the previous section determine
the long-time behavior of the solution to the HHMO-model.  As the
concentration is expected to converge uniformly in parabolic
similarity coordinates, it is convenient to formulate the numerical
scheme directly in $\eta$-$s$ coordinates.  We use simple implicit
first-order time-stepping for the concentration field and direct
propagation of the precipitation function along its characteristic
lines $x=\text{const}$ which transform to hyperbolic curves in the
$\eta$-$s$ plane.  Details of the scheme are provided in
Appendix~\ref{a.scheme}.

The observed behavior is different in the transitional and in the
supercritical regime.  In the transitional regime, the source term is
too weak to maintain precipitation outside of a bounded region on the
$x$-axis, which transforms into a precipitation region which gets
squeezed onto the $s$-axis as time progresses in $\eta$-$s$
coordinates.  In this regime, the asymptotic profile is always
$\Phi_0$; a particular example is shown in Figure~\ref{f.1}.  Note
that the concentration peak drops well below the precipitation
threshold as time progresses.

Figure~\ref{f.2} shows the long-time behavior of the concentration in
the supercritical case.  In this case, the limit profile is
$\Phi_\gamma$, where $\gamma$ is determined as a function of $\alpha$,
$\beta$, and $u^*$ by the solvability condition of
Theorem~\ref{t.selfsimilar}.  The convergence is very robust with
respect to compactly supported changes in the initial condition (not
shown).  We note that the evolution equation in $\eta$-$s$ coordinates
is singular at $s=0$, so the initial value problem is only
well-defined when the initial condition is imposed at some $s_0>0$.
For the numerical scheme, however, there is no problem initializing at
$s=0$.

Along the line $\eta=\alpha$, equivalent to the parabola
$t = x^2/\alpha^2$, on which the source point moves, the concentration
is converging toward the critical concentration $u^*$.  At the same
time, the weighted average of the concentration,
\begin{equation}
  h(x) = \frac1x \int_0^x \xi^2 \, p^*(\xi) \, \d \xi \,,
\end{equation}
is converging to $\gamma$ as $t \to \infty$ or, equivalently,
$x \to \infty$.  This behavior is clearly visible in Figure~\ref{f.3},
where convergence in $h$ is much slower than convergence in $u$.
Figure~\ref{f.3} also shows that grid effects become increasingly
dominant as time progresses.  This is due to the fact that
precipitation occupies at least one full grid cell on the line
$\eta=\alpha$.  However, to be consistent with the conjectured
asymptotics, the temporal width of the precipitation region needs to
shrink to zero.  In the discrete approximation, it cannot do this,
resulting in oscillations of the diagnostics with increasing
amplitude.  For even larger times, the simulation eventually breaks
down completely.  This behavior can be seen as a manifestation of
``rattling,'' described by Gurevich and Tikhomirov
\cite{GurevichT:2017:RattlingSD,GurevichT:2018:SpatiallyDR} in a
related setting.  Here, the scaling of the problem and of the
computational domain leads to an increase of the rattling amplitude
with time.

On any fixed finite interval of time, the amplitude of the grid
oscillations vanishes as the spatial and temporal step sizes go to
zero.  However, it is impossible to design a code in which this
behavior is uniform in time so long as the precipitation function
takes only binary values, i.e., strictly follows condition
\eqref{e.hhmo-p-weak-alternative}.

\section{Long-time behavior of a linear auxiliary problem}
\label{s.auxiliary}

In this section, we study the nonautonomous linear system
\begin{subequations}
  \label{e.auxiliary}
\begin{gather}
  u_t = u_{xx} +
        \frac{\alpha \beta}{2 \sqrt t} \, \delta (x - \alpha \sqrt{t})
        - p \, u \,,
  \label{e.auxiliary.a} \\
  u_x(0,t) = 0 \quad \text{for } t > 0 \,, \label{e.auxiliary.b} \\
  u(x,0) = 0 \quad \text{for } x>0 \label{e.auxiliary.c}
\end{gather}
\end{subequations}
on the space-time domain $\R_+\times\R_+$.  The equations coincide
with the HHMO-model (\ref{e.original.a}--c).  Here, however, we
consider the precipitation function $p(x,t)$ as given, not necessarily
related to $u$ in any way.  The goal of this section is to give
conditions on $p$ such that the solution $u$ converges uniformly in
parabolic similarity coordinates to one of the profiles $\Phi_\gamma$
defined in Section~\ref{s.self-similar}. 

Throughout, we assume that $p \in \mathcal A$, where
\begin{multline}
  \mathcal A
  = \{ p\in L^1_\loc((0,\infty) \times [0,\infty)) \colon \\
       \supp p\cap(\R_+\times[0,T])
       \text{ is compact for every } T>0 \} \,.
  \label{e.admissible}
\end{multline}
In addition, we will assume that $p$ is non-zero, non-negative,
non-decreasing in time, and satisfies condition (P) stated in the
introduction.  In all of the following, we manipulate the equation
formally as if the solution was strong.  A detailed verification that
all steps are indeed rigorous can be found in
\cite[Appendix~B]{Darbenas:2018:PhDThesis}; these result can be
transformed into similarity variables as in Appendix~\ref{a.weak}
below.  In this context, the condition on the support of $p$ in
\eqref{e.admissible} eases the justification of the exchange of
integration and time-differentiation.  More generality is clearly
possible, but this simple assumption covers all cases we need for the
purpose of this paper.

For technical reasons, we distinguish two cases which require
different treatment.  In the first case, $p$ is assumed bounded.  It
is then easy to show that there exists a weak solution
\begin{gather}
  u - \psi \in C^{1,0}(\R_+\times \R_+)\cap L^\infty(\R_+\times\R_+) \,,
  \label{e.weak-linear-class} 
\end{gather}
satisfying \eqref{weak.sol.def.eq}, where $\psi$ is the solution of
the precipitation-less equation given by \eqref{self.similar.p0}; see,
e.g., \cite[Appendix B]{Darbenas:2018:PhDThesis}.

In the second case, $p$ may be unbounded.  In general, the existence
of solutions is then not obvious, so that we assume a solution with
\begin{gather}
  u - \psi \in C^{1,0}(\R_+ \times \R_+)\cap
  W^{1,1}_{2,\loc}(\R_+ \times \R_+)
  \label{e.weak-linear-class-unbounded} 
\end{gather}
exists, and that this solution satisfies the bound
\begin{gather}
  0\le u\le \phi_0
  \quad \text{provided} \quad
  \int_0^{\infty} p^*(x) \, \d x= \infty \label{u.bounded.infty}
\end{gather}
with $\phi_0$ given by \eqref{e.phi.gamma.x-t}, or
\begin{gather}
  0\le u\le \psi
  \quad \text{provided} \quad
  \int_0^{\infty} p^*(x) \, \d x < \infty \,. \label{u.bounded}
\end{gather}
We remark that when $p$ is bounded, it is easy to prove that solutions
$u$ which decay as $x \to \infty$ satisfy the weaker bound
\eqref{u.bounded}.

\begin{remark}
\label{imposing.the.bound} 
Here we will explain why we impose \eqref{u.bounded.infty}.
Proceeding formally, we fix $0<t_0<t_1$ and
$0< x_1< \alpha \sqrt{t_0}$, multiply \eqref{e.auxiliary.a} by $u$,
integrate over $[0,x_1] \times [t_0,t_1]$, and note that the domain of
integration is away from the location of the source, so that
\begin{align}
  \int_{0}^{x_1} u^2 \, \d x \bigg|^{t=t_1}_{t=t_0}
  & = 2 \int_{t_1}^{t_2}u_x \, u \, \d t \bigg|_{x=0}^{x=x_1}
      \notag \\
  & \quad 
    - 2 \int_{t_0}^{t_1} \int_{0}^{x_1} u_x^2 \, \d x \, \d t
    - 2 \int_{0}^{x_1} p^*(x) \int_{t_0}^{t_1} u^2 \, \d t \, \d x \,.
\end{align}
As $u$ and $u_x$ are continuous on the domain of integration, the
first three integrals are finite.  Thus, the last integral must be
finite, too.  When $p^*$ is not integrable near zero, this can only be
true when $u(0,t)=0$ for all $t>0$.  Now note that $\phi_0$ satisfies
\eqref{e.auxiliary.a} for $p\equiv0$ with Dirichlet boundary
conditions
\begin{subequations}
\begin{gather*}
  \phi_0(0,t) = 0 \quad \text{for } t > 0 \,,\\
  \phi_0(x,0) = \lim_{\substack{y\to x\\t\searrow0}} \phi_0(y,t) = 0
  \quad \text{for } x > 0 \,,
\end{gather*}
\end{subequations}
Thus, $\phi_0$ is the natural supersolution for $u$ when $p^*$ is not
integrable. 
\end{remark}

\begin{lemma}
\label{lemma.p.non-decreasing}
Let $p\in \mathcal A$ be non-negative and non-decreasing in time $t$.
Let $u$ be a weak solution to \eqref{e.auxiliary}.  Then $u-\psi$ is
non-increasing in time $t$.
\end{lemma}

\begin{proof}
The proof of \cite[Lemma~3.3]{HilhorstHM:2009:MathematicalSO} applies
literally.  We remark that the result in
\cite{HilhorstHM:2009:MathematicalSO} is stated for solutions to the
HHMO-model, but its proof depends only on the assumption that $p$ is
non-decreasing in $t$ and applies here as well.
\end{proof}

\begin{lemma}
\label{l.u0}
Suppose $p \in \mathcal A$ is non-zero, nonnegative, and satisfies
condition \textup{(P)}.  Let $u$ be a weak solution to
\eqref{e.auxiliary}.  Then $u(0,t) \to 0$ as $t \to \infty$.
\end{lemma}

\begin{remark}
This lemma can be applied to weak solutions of the HHMO-model
\eqref{e.original} provided $u^*<\Psi(\alpha)$ under the additional
assumption that \eqref{p.property} is satisfied.  Then, by
\cite[Lemma~3.5]{HilhorstHM:2009:MathematicalSO}, there is at least
one non-degenerate precipitation region and the assumptions of the
lemma apply.
\end{remark}

\begin{proof}
We construct a supersolution to $u$ as follows.  Fix any $y^*>0$ such
that the support of $p^*$ intersects $[0,y^*]$ on a set of positive
measure.  Define $t^*=(y^*/\alpha)^2$ and
\begin{equation}
  p^r(x,t)
  = \begin{cases}
      \min\{p^*(\abs{x}), 1\}
      & \text{if } x\in[-y^*,y^*] \text{ and } t\ge x^2/\alpha^2 \,, \\
      0 & \text{otherwise} \,. \\
    \end{cases}
  \label{e.pr}
\end{equation}
Let $u^r$ denote the unique bounded weak solution to
\eqref{e.auxiliary} with $p=p^r$ and extend $u^r$ to the left half-plane by even reflection.  Due to
the subsolution principle, $0\le u\le u^r$.  Our goal is to show that
$u^r(0,t)\to 0$ as $t\to\infty$.  We reflect $u^r$ evenly with respect
to $x=0$ axis.  Note that $p^r$ fulfills the conditions of
Lemma~\ref{lemma.p.non-decreasing}.  Therefore, $u^r$ is
non-increasing in $t$ on $[-y^*,y^*]\times[t^*,\infty)$ so that
\begin{equation}
\label{C.bar}
  \lim_{t \to \infty} \inf_{x\in[-y^*,y^*]}u^r(x,t) \equiv c
\end{equation}
exists.  We now express $u^r(0,t)$ for $t>t^*$ via the Duhamel
formula, bound $u^r$ from below by $c$, note that $\HK(-y,t-s)$ is a
decreasing function in $y$, and recall that $p^r$ is supported on
$\{\tau \geq t^*\}$ to estimate
\begin{align}
   u^r(0,t)
   & = \psi(0,t)-\int_0^t\int_{-y^*}^{y^*} \HK(-y,t-\tau) \,
         p^r(y) \,  u^r(y,\tau) \, \d y \, \d \tau
       \notag \\
   & \le \Psi(0)- c \int_{-y^*}^{y^*} p^r(y) \,\d y
       \int_{t^*}^t \HK(y^*,t-\tau) \, \d \tau \,.
   \label{e.ur-estimate}
\end{align}
Changing variables $\tau\to t\tau^\prime$ in the second integral on the
right, we find that
\begin{align}
  \int_{t^*}^t \HK(y^*,t-\tau) \, \d \tau
  & = \sqrt t \int_{t^*/t}^1
      \frac1{\sqrt{4\pi(1-\tau^\prime)}} \,
      \e^{-\tfrac{{y^*}^2}{4t(1-\tau^\prime)}}\,\d \tau^\prime
    \notag \\
  & \sim \sqrt t \int_0^1\frac1{\sqrt{4\pi(1-\tau^\prime)}}\,\d \tau^\prime
    = \sqrt{\frac{t}\pi}
\end{align}
as $t \to \infty$.  This implies $c=0$ as otherwise
$u^r(0,t)\to-\infty$ as $t\to\infty$.  Then the Harnack inequality for
the function $u^r$ on some spatial domain containing the interval
$[-y^*,y^*]$ implies that for any fixed $\delta>0$ there exists a constant $C_\delta>0$ such that
\begin{equation}
\label{u.conv}
  u^r(0,t)
  \le \sup_{y\in[-y^*,y^*]}u^r(y,t)
  \le C_\delta \, \inf_{y\in[-y^*,y^*]} u^r(y,t+\delta)
  \to 0 \text{ as } t \to \infty \,;
\end{equation}
see, e.g., \cite[Section~7.1.4.b]{Evans:2010:PartialDE} and
\cite{Lieberman:1996:SecondOP}.  Hence, $u(0,t)\to0$ as well.
\end{proof}

\begin{lemma}
\label{peak.lemma}
Let $p \in \mathcal A$ be non-negative and non-decreasing in time $t$.
Let $u$ a bounded weak solution to \eqref{e.auxiliary} where, as
before, we write $u(x,t) = v(x/\sqrt{t},\sqrt{t})$.  Then for every
$d>0$ and $\gamma\ge0$, the following is true.
\begin{enumerate}[label=\textup{(\alph*)}]

\item\label{peak.i} There exists $\omega\in(0,1)$ such that for every
$(\eta,s)$ with $v(\eta, s)-\Phi_\gamma(\eta)\ge d$,
\begin{equation}
  \min_{s'\in[\omega s,s]} \max_{\eta\in\R_+}
  \{ v(\eta, s')-\Phi_\gamma(\eta)\} \ge d/2 \,.
   \label{e.peak.i}
\end{equation}

\item\label{peak.ii} There exists $\omega\in(1,\infty)$ such that
for every $(\eta,s)$ with $v(\eta, s)-\Phi_\gamma(\eta)\le -d$,
\begin{equation}
  \max_{s'\in[s,\omega s]} \min_{\eta\in\R_+}
 \{v(\eta, s')-\Phi_\gamma(\eta) \}\le -d/2 \,.
\end{equation}
\end{enumerate}
\end{lemma}

\begin{proof}
Set $V(\eta)=\Psi(\eta)-\Phi_\gamma(\eta)$.  By direct inspection, we
see that $V$ is strictly decreasing on $\R_+$.  In case \ref{peak.i},
\begin{equation}
  d \le v(\eta, s)-\Phi_\gamma(\eta)\le V(\eta) \,.
\end{equation}
Therefore, the possible values of $\eta$ for which the assumption of
case \ref{peak.i} can be satisfied are bounded from above by some
$\eta^*=\eta^*(d,\gamma)$.  By the mean value theorem, for
$\omega \in (0,1)$,
\begin{equation}
  V(\eta) - V(\eta/\omega)
  \leq \max_{\xi \in [\eta,\eta/\omega]} \lvert V'(\xi) \rvert \,
       \Bigl( \frac\eta{\omega} - \eta \Bigr)
  \leq \eta^* \, \max_{\xi \in [0, \eta^*/\omega]} \lvert V'(\xi) \rvert \,
       \frac{1-\omega}\omega
  \leq \frac{d}2
  \label{e.Vdiff}
\end{equation}
where, in the last inequality, $\omega$ has been fixed sufficiently
close to $1$.  This choice is independent of $\eta$.  Now recall that
$t=s^2$ and $x=\eta s$.  Choose any $s'\in[\omega s,s]$, set
$t'={s'}{\vphantom s}^2$ and $\eta'=x/s'$ so that
$\eta' \le \eta/\omega$.  Then
\begin{align}
  d - (u(x,t') - \phi_\gamma(x,t'))
  & \leq (u(x,t) - \phi_\gamma(x,t)) - (u(x,t') - \phi_\gamma(x,t'))
    \notag \\
  & = (u(x,t) - \psi(x,t)) - (u(x,t') - \psi(x,t'))
      + V(\eta) - V(\eta')
    \notag \\
  & \leq V(\eta) - V(\eta/\omega) \leq d/2 \,,
\end{align}
where the first inequality is due to the assumption of case
\ref{peak.i}, the second inequality is due to
Lemma~\ref{lemma.p.non-decreasing} which states that $u-\psi$ is
non-increasing in $t$ for $x$ fixed.  We further used monotonicity of
$V$ in the second inequality.  The last inequality is due to
\eqref{e.Vdiff}.  Altogether, we see that
\begin{equation}
  v(\eta', s') - \Phi_\gamma(\eta') 
  = u(x,t') - \phi_\gamma(x,t')
  \geq d/2 \,.
\end{equation}
This proves \eqref{e.peak.i}.  The proof in case \ref{peak.ii} is
similar.  Notice that
\begin{equation}
  v(\eta, s) \le \Phi_\gamma(\eta) - d < \Phi_\gamma(\eta).
\end{equation}
Therefore, the possible values of $\eta$ for which the assumption of
case \ref{peak.ii} can be satisfied are bounded from below by some
$\eta^*=\eta^*(d,\gamma)>0$. The rest of the proof is obvious.
\end{proof}

\begin{figure}
\centering
\includegraphics{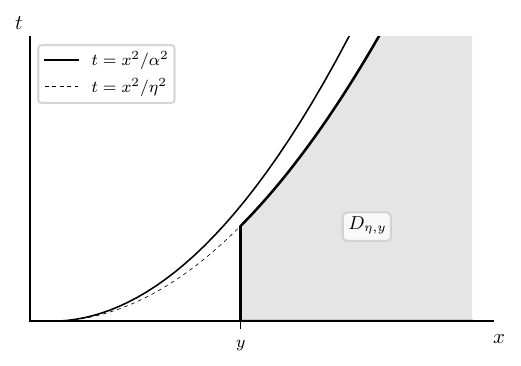}
\caption{Sketch of the region $D_{\eta,y}$ when $\eta>\alpha$.}
\label{f.4}
\end{figure}

In the following, for positive real numbers $\eta$, $y$, and $T$, we
define
\begin{gather}
  D_{\eta,y}
  = \{(x,t) \colon x\ge y, 0\le t\le \eta^{-2}x^2\} \,.
  \label{D.eta.y.infty}
\end{gather}
See Figure~\ref{f.4} for an illustration.

\begin{theorem}
\label{t.linear-asymptotics}
Let $p \in \mathcal A$ be non-zero, non-negative, non-decreasing in
time, and satisfy condition \textup{(P)}.  Assume further for each
$\eta>\alpha$ there exists $y=y(\eta)$ such that $p\equiv0$ on
$D_{\eta,y}$, and that there exists $\gamma\ge 0$ such that
\begin{equation}
  \lim_{x \to \infty} x \int^{\infty}_x p^*(\xi) \, \d \xi = \gamma \,,
  \label{e.p-assumption}
\end{equation}
where $p^*$ denotes the values of $p$ along the line $\eta=\alpha$ as
defined in condition \textup{(P)}.

Let $u$ be a weak solution of class \eqref{e.weak-linear-class} to the
linear non-autonomous equation \eqref{e.auxiliary} with $p$ fixed as
stated.  If $p$ is unbounded, assume further that $u$ is of class
\eqref{e.weak-linear-class-unbounded} and satisfies the bounds
\eqref{u.bounded.infty} or \eqref{u.bounded}.  Then $u$ converges
uniformly to $\Phi_\gamma$.
\end{theorem}

\begin{proof}
Set $w=v-\Phi_\gamma$.  Subtracting \eqref{e.ode2} from \eqref{e.v-a}
and noting that, by assumption, $q(\eta,s)=p^*(s \eta)$ for $\eta <
\alpha$, we obtain
\begin{subequations}
\label{w.equat}
\begin{multline}
  \tfrac12 \, s \, w_s - \tfrac12 \, \eta \, w_{\eta}
  = w_{\eta\eta} - s^2 \, q(\eta,s) \, w \\
    + \Bigl( \frac\gamma{\eta^2} - s^2 \, p^*(s\eta) \Bigr) \,
      \Phi_\gamma \, H(\alpha-\eta)
    - s^2 \, q(\eta,s) \, \Phi_\gamma \, H(\eta-\alpha)
  \label{e.w.a}
\end{multline}
with assumed bounds on $w$, namely
\begin{gather}
  -\Phi_\gamma\le w\le \Phi_0-\Phi_\gamma
  \quad \text{provided} \quad
  \int_0^{\infty} p^*(x) \, \d x= \infty
  \label{e.bound-case1}
\intertext{or}
  -\Phi_\gamma\le w\le \Psi-\Phi_\gamma
  \quad \text{provided} \quad
  \int_0^{\infty} p^*(x) \, \d x < \infty \,.
  \label{e.bound-case2}
\end{gather}
\end{subequations}

To avoid boundary terms
when integrating by parts, we introduce a fourth-power function with
cut-off near zero which is defined, for every $\eps>0$, by 
\begin{equation}
  J_\eps(z) = \begin{cases}
             0 & \text{if } |z|<\eps \,, \\
             (|z|-\eps)^4 & \text{if } |z|\ge \eps \,.
           \end{cases}
\end{equation}
$J_\eps$ is at least twice continuously differentiable, even, positive,
and strictly convex on $(\eps,\infty)$.  We now consider the cases when
$p^*$ is integrable and $p^*$ is not integrable separately.

\begin{case}[$p^*$ is not integrable on $\R_+$]

In this case, we have the bound \eqref{e.bound-case1}, so that
$|w| \le \Phi_0+\Phi_\gamma$. Hence, for $\eps>0$, arbitrary but fixed in
the following, there are $\eta_0=\eta_0(\eps)$ and $\eta_1=\eta_1(\eps)$
with $0<\eta_0<\eta_1<\infty$ such that $|w|\le \eps$, hence
$J_\eps(w) = 0$, for $\eta \notin (\eta_0,\eta_1)$ and all $s>0$.

We multiply \eqref{e.w.a} by $J_\eps^\prime(w)$, integrate on $\R_+$, and
examine the resulting expression term-by-term.  The contribution from
the first term reads
\begin{equation}
  \label{eq.r.1}
  \frac12 \int_0^\infty s \, w_s \, J^\prime_\eps(w) \, \d\eta
  = \frac{s}2 \, \frac{\d}{\d s}
      \int_{\eta_0}^\infty J_\eps(w) \, \d\eta 
\end{equation}
and the second term contributes
\begin{align}
  \frac12 \int_0^\infty \eta \, w_{\eta} \, J^\prime_\eps(w) \, \d\eta
  & = \frac12 \int_{\eta_0}^{\eta_1} \eta \, w_{\eta} \,
        J^\prime_\eps(w) \, \d\eta
     \notag \\
  & = -\frac12 \int_{\eta_0}^{\eta_1} J_\eps(w) \, \d\eta
    = -\frac12 \int_{\eta_0}^\infty J_\eps(w) \, \d\eta \,.
  \label{eq.l.2}
\end{align}
Combining both expressions, we obtain
\begin{equation}
  \label{comb.2}
  \frac{s}2 \, \frac{\d}{\d s} \int_{\eta_0}^\infty J_\eps(w) \, \d\eta
  + \frac12 \int_{\eta_0}^\infty J_\eps(w) \, \d\eta
  = \frac{\d}{\d s}
    \biggl(
      \frac{s}2 \int_{\eta_0}^\infty J_\eps(w) \, \d\eta
    \biggr) \,.
\end{equation}
The contribution from the first term on the right of \eqref{e.w.a}
reads
\begin{align}
  \int_0^\infty w_{\eta\eta} \, J^\prime_\eps(w) \, \d\eta
  & = \int_{\eta_0}^{\eta_1}w_{\eta\eta} \, J^\prime_\eps(w) \, \d\eta
      \notag \\
  & = - \int_{\eta_0}^{\eta_1}w_\eta^2 \, J^{\prime\prime}_\eps(w) \, \d\eta
    = - \int_{\eta_0}^{\infty}w_\eta^2 \, J^{\prime\prime}_\eps(w) \, \d\eta\,.
  \label{eq.l.1}
\end{align}
The contribution from the second term on the right of \eqref{e.w.a}
satisfies
\begin{equation}
  \label{eq.l.3}
  -\int_0^\infty s^2 \, q(\eta,s) \, w \, J^\prime_\eps(w) \, \d\eta
  \leq 0
\end{equation}
because the product $w \, J^\prime_\eps(w)$ is clearly non-negative.  To
investigate the contribution coming from the third term on the right
of \eqref{e.w.a}, we integrate by parts, so that
\begin{align}
  & \int_0^\infty 
        \Bigl( \frac{\gamma}{\eta^2} - s^2 \, p^*(s\eta) \Bigr) \,
        \Phi_\gamma \, H(\alpha-\eta) \, J^\prime_\eps(w) \, \d\eta
    \notag \\
  & = \int_{\eta_0}^\alpha 
        \Bigl( \frac{\gamma}{\eta^2} - s^2 \, p^*(s\eta) \Bigr) \,
        \Phi_\gamma \, J^\prime_\eps(w) \, \d\eta
    \notag \\
  & = g(\alpha,s) \, \Phi_\gamma(\alpha) \, J^\prime_\eps(w(\alpha,s))
      - \int_{\eta_0}^\alpha g \, \Phi^\prime_\gamma \, 
        J^\prime_\eps(w) \, \d\eta
      - \int_{\eta_0}^\alpha g \, \Phi_\gamma \, 
        w_\eta \, J^{\prime\prime}_\eps(w) \, \d\eta \,,
  \label{eq.l.4}
\end{align}
where $g$ is an anti-derivative of the term in parentheses, namely
\begin{equation}
  g(\eta,s)
  = s^2\int_{\eta}^\infty p^*(s\kappa)\, \d\kappa
    - \frac\gamma\eta
  = s\int_{s\eta}^\infty p^*(\zeta) \, \d\zeta
    - \frac\gamma\eta \,.
\end{equation}
We note that
for fixed $\eta>0$, due to \eqref{e.p-assumption}, $g(\eta, s)\to0$ as
$s\to\infty$.

When $\eps<u^*$, the equation $\Phi_0(\eta)=\eps$ has one root
$\eta \leq \alpha$.  Since $u \leq \Phi_0$, we can set $\eta_0=\eta$
so that $\eta_0 \leq \alpha$, which we assume henceforth.  Combining
\eqref{eq.l.1} with the last term in \eqref{eq.l.4}, we obtain
\begin{align}
  - \int_{\eta_0}^\infty & w_{\eta}^2 \, J^{\prime\prime}_\eps(w) \, \d\eta
      - \int_{\eta_0}^\alpha g \, \Phi_\gamma \, w_\eta \, 
          J^{\prime\prime}_\eps(w) \, \d\eta
    \notag \\
  & = - \int_{\alpha}^{\infty} w_\eta^2 \, J^{\prime\prime}_\eps(w) \, \d\eta
      - \int_{\eta_0}^{\alpha} \bigl( w_\eta + \tfrac12 \, g \, 
           \Phi_\gamma \bigr)^2 \, J^{\prime\prime}_\eps(w) \, \d\eta
    \notag \\
  & \quad + \frac14 \int_{\eta_0}^\alpha g^2 \,
           \Phi_\gamma^2 \, J^{\prime\prime}_\eps(w) \, \d\eta
    \notag \\
  & = \frac12 \int_{\eta_0}^\alpha g^2 \,
           \Phi_\gamma^2 \, J^{\prime\prime}_\eps(w) \, \d\eta - G^*(s)
  \label{comb.1}
\end{align}
where
\begin{align}
  G^*(s)
  & = \int_{\alpha}^{\infty} w_\eta^2 \, J^{\prime\prime}_\eps(w) \, \d\eta
      + \int_{\eta_0}^{\alpha}
          \bigl( w_\eta+ \tfrac12 \, g \, \Phi_\gamma \bigr)^2 \,
          J^{\prime\prime}_\eps(w) \, \d\eta
      + \frac14 \int_{\eta_0}^{\alpha} 
          g^2 \, \Phi_\gamma^2 \, J^{\prime\prime}_\eps(w) \, \d\eta
      \notag\\
  & \ge \int_{\alpha}^{\infty} w_\eta^2 \, J^{\prime\prime}_\eps(w) \, \d\eta
      + \frac12 \int_{\eta_0}^{\alpha} w_\eta^2 \,
          J^{\prime\prime}_\eps(w) \, \d\eta
      \notag\\
  & \ge \frac12 \int_0^{\infty} w_\eta^2 \,
      J^{\prime\prime}_\eps(w) \, \d\eta \,.
  \label{e.gstar}
\end{align}
We note that we have used the Jensen inequality in the first
inequality of this lower bound estimate.

Finally, the last term on the right of \eqref{e.w.a} is treated as
follows.  We define
\begin{equation}
  F(s)
  = s^2 \int_0^\infty q(\eta,s) \, \Phi_\gamma \, J_\eps'(w) \,
    H(\eta-\alpha) \, \d\eta
\end{equation}
and
\begin{equation}
  \Gamma(x)=x\int_x^\infty p^*(\xi)\,\d \xi \,.
\label{Gamma.def}
\end{equation}
For fixed $\eta^*>\alpha$ we can find $y=y(\eta^*)$ such that
$p\equiv0$ on $D_{\eta^*,y}$, i.e., $q(\eta, s)=0$ for all
$\eta>\eta^*$ and $s\ge y/\eta^*$. 

Then, for
$s\ge s_0 \equiv y/\eta^*$,
\begin{align}
   F(s) \le \Phi_\gamma(\alpha) \, J_\eps'(\Psi(\alpha))
     \int_\alpha^{\eta^*} s^2 \, q(\eta,s) \, \d\eta \,.
   \label{e.F-bound1}
\end{align}
Since $p$ is non-decreasing in time $t$, we estimate
\begin{equation}
  \int_\alpha^{\eta^*} s^2 \, q(\eta, s) \, \d\eta
  \leq \int_\alpha^{\eta^*} s^2 \, p^*(\eta s) \, \d\eta
  = s \int_{s\alpha}^{s\eta^*} p^*(\kappa) \, \d\kappa
  = \frac{\Gamma(s\alpha)}{\alpha}
    - \frac{\Gamma(s\eta^*)}{\eta^*} \,.
\end{equation}
Inserting this bound into \eqref{e.F-bound1} and noting that, due to
\eqref{e.p-assumption}, $\lim_{x\to\infty}\Gamma(x)=\gamma$, we find
that
\begin{equation}
  \limsup_{s \to \infty} F(s)
  \le \Phi_\gamma(\alpha) \, J_\eps'(\Psi(\alpha)) \,
      \Bigl( \frac\gamma\alpha - \frac\gamma{\eta^*} \Bigr) \,.
  \label{e.F-bound2}
\end{equation}

Adding up the individual contributions and neglecting the clearly
non-positive terms on the right hand side as indicated, we obtain
altogether
\begin{equation}
  \frac{\d}{\d s}
    \biggl( \frac{s}2 \int_0^\infty J_\eps(w) \, \d\eta \biggr)
  \leq G(s) - G^*(s) + F(s)
  \label{e.diff-ineq}
\end{equation}
with
\begin{multline}
  G(s)
  = g(\alpha,s) \, \Phi_\gamma(\alpha) \, J^\prime_\eps(w(\alpha,s)) 
    - \int_{\eta_0}^\alpha g \, \Phi^\prime_\gamma
        J^\prime_\eps(w) \, \d\eta 
    + \frac12 \int_{\eta_0}^\alpha
        g^2 \, \Phi_\gamma^2 \, J^{\prime\prime}_\eps(w) \, \d\eta \,.
\label{G.exp}
\end{multline}
We note that $G(s) \to 0$ as $s \to 0$.  Indeed, the first term
converges to zero because $g$ converges to zero.  The two integrals
converge to zero because, in addition, on $[\eta_0, \alpha]$ the
function $g$ satisfies the uniform bound
\begin{equation}
  g(\eta,s)
  = \frac{\Gamma(\eta s) - \gamma}\eta
  \leq \frac1\eta \, \sup_{x \geq \eta_0 s_0} \Gamma(x)
\end{equation}
which, together with the known bounds on $\Phi$, $\Phi'$, and $w$,
implies that the dominated convergence theorem is applicable.  Hence,
each of the integrals converges to zero.

Integrating \eqref{e.diff-ineq} from $s_0$ to $s$, we obtain
\begin{equation}
  \int_0^\infty J_\eps(w(\eta,s)) \, \d\eta
  - \frac{s_0}s \int_0^\infty J_\eps(w(\eta,s_0) \, \d\eta
  \leq \frac{2}s \int_{s_0}^s
       \bigl( G(\sigma) - G^*(\sigma) + F(\sigma) \bigr) \,
       \d \sigma \,.
  \label{e.integrated}
\end{equation}
We now take $\limsup_{s\to\infty}$.  The second term on the left
vanishes trivially.  Since $G$ converges to zero, so does its time
average, so its contribution is negligible in the limit.  $G^*$ is
non-negative, hence can be neglected.  For the contribution from $F$,
we refer to \eqref{e.F-bound2}.  Hence,
\begin{equation}
  \limsup_{s\to+\infty}\int_0^\infty J_\eps(w) \, \d\eta
  \le 2 \Phi_\gamma(\alpha) \, J_\eps'(\Psi(\alpha)) \,
      \Bigl( \frac\gamma\alpha - \frac\gamma{\eta^*} \Bigr) \,.
\label{limsup.J_l.F}
\end{equation}
Since $\eta^*>\alpha$ is arbitrary, we conclude that
\begin{equation}
  \lim_{s\to\infty} \int_0^\infty J_\eps(w(\eta,s)) \, \d\eta = 0 \,.
  \label{e.Jl-convergence}
\end{equation}
This implies that for every fixed $\eps>0$,
\begin{equation}
  m(\{ \eta \colon \lvert w \rvert>2\eps\} ) \, J_\eps(2\eps)
  \leq \int_{\{ \eta \colon \lvert w \rvert>2\eps\}} J_\eps(w)\,\d \eta
  \leq \int_0^\infty J_\eps(w) \, \d \eta \to 0
\label{J_l.in.measure}
\end{equation}
as $s\to\infty$, where $m$ is the Lebesgue measure on the real line,
i.e., $\lvert w \rvert$ converges to zero in measure.  Due to the
bound on $w$, the dominated convergence theorem with convergence in
measure, e.g.\ \cite[Corollary~2.8.6]{Bogachev:2007:MeasureT}, implies
that $v\to\Phi_\gamma$ in $L^r(\R_+)$ for every $r \in [1,\infty)$.
\end{case}

\begin{case}[$p^*$ is integrable on $\R_+$]
\label{c.2}
When $p^*$ is integrable, we only have the weaker bound on $w$ given
by \eqref{e.bound-case2}.  Thus, we must take $\eta_0=0$.  On the
other hand, due to Lemma~\ref{l.u0}, $u(0,s)$ is converging to $0$ as
$s\to\infty$.  Thus, we fix $\eps>0$ and choose $s_0=s_0(\eps)$ satisfying
$u(0,s_0)<\eps$.  Then $J_\eps(w(0,s))=J_\eps^\prime(w(0,s))=0$ for all $s>s_0$
so that the boundary terms when iterating by parts in \eqref{eq.l.1},
\eqref{eq.l.2}, and \eqref{eq.l.4} vanish as before, so that all
computations from Case~1 up to equation \eqref{G.exp} remain valid as
before.

The bound on $g$ now takes the form 
\begin{equation}
  g(\eta,s)
  = \frac{\Gamma(\eta s) - \gamma}\eta
  \leq \frac1\eta \, \sup_{x \geq 0} \Gamma(x)
\end{equation}
where, as before, $\Gamma$ is given by \eqref{Gamma.def}.  This
implies that the integrands in the second and third term in
\eqref{G.exp} satisfy bounds on the interval $[0,\alpha]$ which take
the form
\begin{subequations}
\begin{gather}
  \lvert g \, \Phi^\prime_\gamma \, J^\prime_\eps(w)\rvert
  \le C_1 \, \eta^{\kappa-2} \,,
  \label{G.2nd.member.bound} \\
  \lvert g^2 \, \Phi_\gamma^2 \, J^{\prime\prime}_\eps(w)\rvert
  \le C_2 \, \eta^{2(\kappa-1)} \,,
  \label{G.3rd.member.bound}
\end{gather}
\end{subequations}
where $C_1$ and $C_2$ are positive constants.  When $\gamma>0$ so that
$\kappa>1$, both bounds are integrable on $[0,\alpha]$ and the
dominated convergence theorem applies as before, proving that
$G(s)\to0$ as $s\to\infty$.  When $\gamma=0$ so that $\kappa=1$, the
second bound \eqref{G.3rd.member.bound} is still integrable, but the
first is not.  Thus, for the second term on the right of
\eqref{G.exp}, we change the strategy as follows.

Observe that when $\gamma=0$, then 
\begin{equation}
  g(\eta,s) = s \int_{s\eta}^\infty p^*(\zeta) \, \d\zeta \ge 0 \,.
\end{equation}
Thus, the second term in \eqref{G.exp} is bounded above by
\begin{equation}
  - \int_0^{\alpha} g(\eta,s) \, \Phi_0'(\eta) \,
      J_\eps'(w(\eta,s)) \, \d \eta
  \le \int_0^{\alpha} \I_{\{w(\eta,s)<0\}}(\eta) \, g(\eta,s) \,
        \Phi_0'(\eta) \, \lvert J_\eps'(w(\eta,s)) \rvert \, \d \eta \,.
  \label{e.g-second-term}
\end{equation}
Note that $w(\eta,s)<0$ if and only if $u(\eta,s)<\Phi_0(\eta)$.
Moreover, $\Phi_0(\eta) = O(\eta)$ as $\eta \to 0$ so that
$\I_{\{w<0\}} \, J_\eps'(w) = O(\eta^3)$.  Altogether, there exists $C_3>0$
such that
\begin{equation}
  \I_{\{w(\eta,s)<0\}}(\eta) \, g(\eta,s) \, \Phi_0'(\eta) \,
  \lvert J_\eps'(w(\eta,s)) \rvert
  \leq C_3 \, \eta^2 
\end{equation}
which provides an integrable upper bound for the integrand on the
right of \eqref{e.g-second-term}.  The dominated convergence theorem
then proves that the integral on the right of \eqref{e.g-second-term}
converges to zero as $s \to \infty$.

Thus, we find in all cases that
$\limsup_{\sigma\to\infty}G(\sigma)\le 0$, so that the argument from
\eqref{e.integrated} to \eqref{e.Jl-convergence} proceeds as before
and \eqref{J_l.in.measure} is valid for every $\eps>0$.  This shows that
$\lim_{s\to\infty}w=0$ in $L^r$.
\end{case}

In the final step, we bootstrap from $L^r$-convergence to uniform
convergence on $\R_+$.  We argue by contradiction and for both cases
at once.

Suppose convergence is not uniform.  Then there exists $d>0$ such that
\begin{equation}
  \limsup_{s\to\infty}\max_{\eta\in\R_+}w(\eta,s) \ge 2d
\end{equation}
or
\begin{equation}
  \liminf_{s\to\infty}\min_{\eta\in\R_+}w(\eta,s) \le -2d \,.
  \label{e.uniform-case2}
\end{equation}
Suppose that the first alternative holds; the argument for the second
alternative proceeds analogously and shall be omitted.  By
Lemma~\ref{peak.lemma}\ref{peak.i}, there exists $\omega \in (0,1)$, a
sequence $s_i \to \infty$, and a sequence $\eta_i$ such that for every
$i \in \N$,
\begin{equation}
  \min_{s \in [\omega s_i,s_i]} w(\eta_i,s) \geq d/2 \,.
\end{equation}
Due to the uniform bound on $w$ which decays as $\eta \to \infty$, the
sequence $\eta_i$ must be contained in a compact interval of length
$L$ (possibly dependent on $d$).  In the following, fix $\eps < d/4$.

For fixed $s \in [\omega s_i,s_i]$, let
\begin{equation}
  \eta_0 = \max \{ \eta < \eta_i \colon J_\eps(w(\eta,s)) = 0 \} \,.
\end{equation}
(By continuity, the maximum exists and is less than $\eta_i$; in
Case~2 we may need to take $i$ large enough so that $\omega s_i > s_0$.)
Due to the fundamental theorem of calculus,
\begin{equation}
  J_\eps^{1/2}(w(\eta_i,s)) - J_\eps^{1/2}(w(\eta_0,s))
  = \int_{\eta_0}^{\eta_i} \partial_\eta
      J_\eps^{1/2} (w(\eta,s)) \, \d \eta 
\end{equation}
so that, noting that $J_\eps^{1/2}(w(\eta_i,s))=0$,
$J_\eps^{1/2}(w(\eta_0,s))\ge (d/2-\eps)^2$, and $d/2-\eps\ge d/4$ on the left
and using the Cauchy--Schwarz inequality on the right, we obtain
\begin{equation}
  \Bigl( \frac{d}4 \Bigr)^2
  \leq (\eta_i - \eta_0)^{\tfrac12} \,
       \biggl(
         \int_{\eta_0}^{\eta_i} 4 \, (\lvert w \rvert - \eps)^2 \,
           w_\eta^2 \, \d \eta
       \biggr)^{\tfrac12}
  \leq \sqrt L \,
        \biggl(
         \frac13 \int_{\R_+} J_\eps^{\prime\prime} (w) \,
           w_\eta^2 \, \d \eta
       \biggr)^{\tfrac12} \,.
\end{equation}
We conclude that the integral on the right is bounded below by some
strictly positive constant, say $b$, which only depends on $d$.  Due
to \eqref{e.gstar}, $b/2$ is also a lower bound on $G^*$.  Thus,
returning to \eqref{e.integrated} with $s=s_i$ and $s_0=\omega s_i$, we
obtain
\begin{align}
  \int_0^\infty J_\eps(w(\eta,s_i)) \, \d\eta
  & - \omega \int_0^\infty J_\eps(w(\eta,\omega s_i) \, \d\eta
     \leq \frac{2}{s_i} \int_{\omega s_i}^{s_i}
       \bigl( G(\sigma) - G^*(\sigma) + F(\sigma) \bigr) \, \d \sigma
    \notag \\
  & \leq
    - (1-\omega) \, b + \frac{2}{s_i} \int_{\omega s_i}^{s_i}
       \bigl( G(\sigma) + F(\sigma) \bigr) \, \d \sigma \,.
\end{align}
We now let $i \to \infty$ and observe that, due to
\eqref{e.Jl-convergence}, the two terms on the left converge to zero.
For the integral on the right, we apply the same asymptotic bounds as
in the first part of the argument, so that 
\begin{equation}
  0
  \le - (1-\omega) \, b + \Phi_\gamma(\alpha) \, J_\eps'(\Psi(\alpha)) \,
      \Bigl( \frac\gamma\alpha - \frac\gamma{\eta^*} \Bigr) \,.
\end{equation}
Since $\eta^*>\alpha$ is arbitrary, we reach a contradiction.
Alternative \eqref{e.uniform-case2} can be argued similarly, with
reference to Lemma~\ref{peak.lemma}\ref{peak.ii}.  This completes the
proof of uniform convergence.
\end{proof}

\begin{remark}
The use of the cut-off function $J_\eps$ is a technical necessity to
avoid boundary terms when integrating by parts.  Our particular choice
of $J_\eps$ amounts essentially to an $L^4$ estimate; the exponent $4$
was chosen purely for the convenience of an easy explicit cut-off
construction.  The implication of $L^r$-convergence for any
$r \in [1,\infty)$ can then be understood as a consequence of
boundedness of $w$ and $L^p$-interpolation.
\end{remark}

\section{Long-time behavior of the HHMO-model}
\label{s.hhmo}

In this section, we turn to studying the long-time behavior of
solutions to the actual HHMO-model \eqref{e.original}.  We first prove
a series of simple results,
Theorems~\ref{self.similar.does.not.exists}--\ref{no.interring.infinite},
which are all based on constructing suitable sub- and supersolutions
whose long-time behavior can be described by
Theorem~\ref{t.linear-asymptotics}.  We then turn to maximum principle
arguments which show that the onset of precipitation in the HHMO-model
is asymptotically close to the line $\eta=\alpha$, so that a statement
like Theorem~\ref{t.linear-asymptotics} also holds true for
HHMO-solutions.  Finally, we prove the main result of this section,
which can be seen as a converse statement, the identification of the
only possible limit profile for the HHMO-model.  The two main
statements are summarized as Theorem~\ref{convergence.conclusion} at
the end of the section.

\begin{theorem}[Long-time behavior in the transitional regime]
\label{self.similar.does.not.exists}
Let $(u,p)$ be a weak solution to \eqref{e.original} in the
transitional regime where $u^*_0<u^*<\Psi(\alpha)$, $u^*_0$ being
defined in \eqref{eq.kappa.2} with $\gamma = 0$.  Then $p(x,t)=0$ for
all $x$ large enough.  Moreover, $u$ converges uniformly to the
profile $\Phi_0$.
\end{theorem}

\begin{proof}
Set $Y=X_1$, the right endpoint of the first precipitation region, see
\cite[Lemma~3.5]{HilhorstHM:2009:MathematicalSO}, provided that it is
finite.  If it were infinite, we would set $Y=1$.  (The theorem shows
that this case is impossible, but at this point we do not know.)  We
then define
\begin{equation}
  p_1(x,t)
  = \begin{cases}
      H(t-x^2/\alpha^2) & \text{for } x \le Y \,,\\
      0 & \text{otherwise} \,.
    \end{cases}
\end{equation}
and note that $p_1\le p$. This function satisfies condition (P) with
$p_1^*(x)=H(Y-x)$ for $x \ge 0$ as well as the conditions of
Theorem~\ref{t.linear-asymptotics}; we note, in particular, that
\begin{equation}
  x \int_x^\infty p_1^*(\xi) \, \d\xi = 0
\end{equation}
for $x \ge Y$ so that \eqref{e.p-assumption} holds with $\gamma=0$.

Let $u_1$ denote the weak solution to the linear non-autonomous
problem \eqref{e.auxiliary} with $p=p_1$.  By construction, $u_1$ is a
supersolution to $u$ and by Theorem~\ref{t.linear-asymptotics}, $u_1$
converges uniformly to $\Phi_0$.  This implies that there exists $T>0$
such that for all $t>T$,
\begin{equation}
  u(x,t) \leq u_1(x,t) 
  \leq \tfrac12 \, (\Phi_0(\alpha)+u^*)
  = \tfrac12 \, (u_0^*+u^*) < u^* \,.
\end{equation}
Further, due to Lemma~\ref{u.psi}, $u(x,t)<u^*$ for all $(x,t)$ with
$x>\alpha^* \sqrt{T}$ and $t \leq T$.  Combining these two bounds, we
find that $u(x,t)<u^*$ for all $x>\alpha^*\sqrt{T}$ and therefore
no ignition of precipitation is possible in this region.

Let $p_2 (x,t)= \I_{[0,\alpha^*\sqrt{T}]}(x)$.  By the same argument
as before, $p_2$ satisfies the conditions of
Theorem~\ref{t.linear-asymptotics} with $\gamma=0$.  Let $u_2$ denote
the solution to \eqref{e.auxiliary} corresponding to $p=p_2$.  Since
$p_2 \geq p$, $u_2$ is a subsolution of $u$.
Theorem~\ref{t.linear-asymptotics} implies that $u_2$ converges
uniformly to $\Phi_0$.  Altogether, as $u_2\le u \le u_1$, we conclude
that $u$ converges uniformly to $\Phi_0$ as well.
\end{proof}

\begin{remark}
A similar argument can be made in case of a marginal precipitation
threshold.  In Theorem~\ref{weak.sol.ustar.threshold.critical}, we
have already seen that marginal solutions are not unique.  For the
long-time behavior, there are two possible cases: If $p$ remains zero
a.e., then $u=\psi$ everywhere, so the long-time profile in $\eta$-$s$
coordinates is $\Psi$.  As soon as spontaneous precipitation occurs on
a set of positive measure, the long-time profile is $\Phi_0$ instead.
To see this, let $c\in(0,1]$ be such value that $p\ge c$ on some
subset of $\R\times\R_+$ of positive measure.  Select $t^*$ such that
$p(\,\cdot\,, t^*)\ge c$ on some subset $A \subset \R$ of positive
measure.  Set $p_1(x,t)=c \, \I_A(x) \, H(t-t^*)$ and let $u_1$ denote
the associated bounded solution to the auxiliary problem
\eqref{e.auxiliary}; $u_1$ is a supersolution for $u$.  Even though
condition (P) does not hold literally, the argument in the proof of
Theorem~\ref{t.linear-asymptotics} still works when restricted to
$s\ge \sqrt{t^*}$.  Hence, $u^1$ converges uniformly to $\Phi_0$.  A
subsolution, also converging to $\Phi_0$, can be constructed as in the
proof of Lemma~\ref{l.u0}.
\end{remark}

The next theorem states that it is impossible to have a precipitation
ring of infinite width in the strict sense that $u$ permanently
exceeds the precipitation threshold in some neighborhood of the source
point.  A similar theorem is stated in \cite[Theorem
3.10]{HilhorstHM:2009:MathematicalSO}, albeit under a certain
technical assumption on the weak solution.  The theorem here does not
depend on this assumption.

\begin{theorem}[No ring of infinite width]
\label{no.ring.infinite}
Let $(u,p)$ be a weak solution to \eqref{e.original}.  Then
\begin{equation}
  \liminf_{x \to \infty} u(x, x^2/\alpha^2) \le u^*
  \label{e.u-claim}
\end{equation} 
and there exist precipitation gaps for arbitrarily large $x$ in the
following sense: for every $Y> 0$,
\begin{equation}
  \essinf_{\substack{x \ge Y \\ t \ge x^2/\alpha^{2}}} p(x,t) < 1 \,.
  \label{e.p-claim}
\end{equation}
\end{theorem}

\begin{proof}
Suppose the converse, i.e., that there exists $Y>0$ such that
$p(x,t)=1$ for almost all pairs $(x,t)$ with $x \geq Y$ and
$t \geq x^2/\alpha^2$.  Choose $\gamma>0$ such that
$\Phi_\gamma(\alpha) < u^*$.  This is always possible because the
argument used in the proof of Theorem~\ref{t.selfsimilar} shows that
$\Phi_\gamma(\alpha) = u_\gamma^* \to 0$ as $\gamma \to \infty$.  Now
increase $Y$ such that $Y \geq \sqrt \gamma$, if necessary, and set
\begin{equation}
  p_3(x,t) =
  \begin{dcases}
    \frac\gamma{x^2} &
    \text{for } x \ge Y \text{ and } t\ge x^2/\alpha^2 \,, \\
    0 & \text{otherwise} \,.
  \end{dcases}
\end{equation}
Then $p_3 \leq p$ and $p_3$ clearly satisfies the assumptions of
Theorem~\ref{t.linear-asymptotics} with the chosen value of $\gamma$.  

Let $u_3$ denote the weak solution to the linear nonautonomous problem
\eqref{e.auxiliary} with $p=p_3$.  By construction, $u_3$ is a
supersolution to $u$ and by Theorem~\ref{t.linear-asymptotics}, $u_3$
converges uniformly to $\Phi_\gamma$.  This implies that there exists
$T>0$ such that for all $t>T$,
\begin{equation}
  u(x,t) \leq u_3(x,t) 
  \leq \tfrac12 \, (\Phi_\gamma(\alpha)+u^*) < u^* \,.
\end{equation}
Further, due to Lemma~\ref{u.psi}, $u(x,t)<u^*$ for all $(x,t)$ with
$x>\alpha^* \sqrt{T}$ and $t \leq T$.  Combining these two bounds, we
find that $u(x,t)<u^*$ for all $x>\alpha^*\sqrt{T}$.  Therefore,
$p\equiv 0$ in this region.  Contradiction.  This proves that
\eqref{e.p-claim} holds true for every $Y>0$.

To prove \eqref{e.u-claim}, assume the contrary, i.e., that
$\liminf_{x \to \infty} u(x, x^2/\alpha^2) > u^*$.  Then there exists
$Y>0$ such that $u(x, x^2/\alpha^2)>u^*$ for all $x\ge Y$, so that
\begin{equation}
  \essinf_{\substack{x \ge Y \\ t \ge x^2/\alpha^{2}}} p(x,t) = 1 \,.
\end{equation}
As this contradicts \eqref{e.p-claim}, the proof is complete.
\end{proof}

In the supercritical regime, we also have the converse: there is no
precipitation gap of infinite width, i.e., the reaction will always
re-ignite at large enough times.  The following theorem mirrors
\cite[Theorem~3.13]{HilhorstHM:2009:MathematicalSO} but does not
require the technical condition assumed there.

\begin{theorem}[No gap of infinite width in the supercritical regime]
\label{no.interring.infinite}
Let $(u,p)$ be a weak solution to \eqref{e.original} in the
supercritical regime where $u^*<u^*_0<\Psi(\alpha)$. Then there is
ignition of precipitation for arbitrarily large $x$ in the following
sense: for every $Y>0$,
\begin{equation}
  \esssup_{\substack{x \ge Y \\ t\in\R_+}} p(x,t) > 0 \,.
  \label{e.nogap}
\end{equation}
\end{theorem}

\begin{proof}
Assume the contrary, i.e., there exists $Y>0$ such that $p=0$ a.e.\ on
$[Y,\infty)\times\R_+$.  We construct the supersolution $u_1$ as in
the proof of Theorem~\ref{self.similar.does.not.exists}.  In
particular, $u_1$ converges uniformly to $\Phi_0$.
 
We set $p_2(x,t)=\I_{[-Y,Y]}(x)$ and let $u_2$ be the associated weak
solution to \eqref{t.linear-asymptotics} with given $p_2$. Since
$p\ge p_2$, $u_2$ is a subsolution of $u$.  Further, $p_2$ satisfies
condition (P) with $p^*_2(x)=\I_{[0,Y]}(x)$.  Hence,
\begin{equation}
  x\int_x^\infty p^*_2(\xi)\,\d\xi = 0\text{ for }x\ge Y
\end{equation}
so that the pair $(u_2,p_2)$ satisfies the conditions of
Theorem~\ref{t.linear-asymptotics} for $\gamma=0$.  Therefore, $u_2$
converges uniformly to $\Phi_0$.

Altogether, $u$ converges uniformly to $\Phi_0$, in particular,
$\lim_{t\to\infty}u(\alpha\sqrt t, t)=\Phi_0(\alpha)=u^*_0>u^*$.  This
contradicts Theorem~\ref{no.ring.infinite}, so \eqref{e.nogap} holds
for every $Y>0$.
\end{proof}


\begin{remark}
\label{r.critical}
Between Theorem~\ref{self.similar.does.not.exists} and
Theorem~\ref{no.interring.infinite}, we cannot say anything about the
critical case when $u_0^*=u^*$.  This case is highly degenerate, so
that both arguments above fail.  We believe that the problem is of
technical nature, i.e., treating the degeneracy in the proof.  We have
no indication that the qualitative behavior is different from the
neighboring cases and conjecture that the asymptotic profile is
$\Phi_0$ as well.
\end{remark}

\begin{lemma}
\label{p=0-on-D.infty}
Let $(u,p)$ be a weak solution to \eqref{e.original}.  Suppose
$\eta\geq \alpha$ and $t_0\geq 0$ are such that
$u(\eta\sqrt t, t) \leq u^*$ for all $t\ge t_0$.  Then
\begin{enumerate}[label={\upshape(\roman*)}]
\item \label{i.p0.i} there exists $z\geq 0$ such that $u<u^*$ and
$p\equiv 0$ in the interior of $D_{\eta, z}$.

\item \label{i.p0.iii} If $\eta=\alpha$ and the bound $u(\eta\sqrt t,
t) < u^*$ for all $t\ge t_0$ holds with strict inequality, then
$u(x,t) < u^*$ and $p(x,t)=0$ for all $x\ge z$ and $t \geq 0$.
\end{enumerate}
\end{lemma}

\begin{proof}
Select $z \geq \eta \sqrt{t_0}$ such that $u(z,t) \leq u^*$ for all
$t\in[0,z^2/\eta^2]$.  This is always possible for otherwise, due to
\eqref{e.hhmo-p-weak-alternative}, the solution $(u,p)$ would have a
ring of infinite width.  By assumption, we also have
$u(x,x^2/\eta^2) < u^*$ for all $x\geq z$.  Since $u(x,0)=0$, the
parabolic maximum principle then implies that $u$ takes its maximum on
the boundary of $D_{\eta, z}$ where it is bounded above by
$u^*$, and that $u<u^*$ anywhere in the interior.  This implies $p=0$
in the interior of $D_{\eta, z}$, so that the proof of case
\ref{i.p0.i} is complete.

(To see how this derives from the standard statement of the maximum
principle, take, for every $x \geq z$, the cylinder
\begin{equation}
  U_x = [x, X(x)] \times [0,x^2/\eta^2] \,,
\end{equation}
where, due to the upper bound $u \leq \psi$ from Lemma~\ref{u.psi}, we
can choose $X(x)$ large enough so that the maximum of $u$ on
$\partial U_x$ does not lie on the right boundary.  Then $u$ takes its
maximum on the parabolic boundary of $U_x$; by construction, the
maximum must lie on the left-hand boundary
$\{ (x,t) \colon 0 \leq t \leq x^2/\eta^2 \}$.  Moreover, as $u$
cannot be a constant, it is strictly smaller than its maximum
everywhere in the interior of $U_x$.  Since $x \geq z$ is arbitrary,
the maximum must lie on any of the left side boundaries which is not
itself an interior point for some other $U_x$.  The set of all such
points is contained in the boundary of $D_{\eta, z}$.)

When $\eta=\alpha$, we recall that, by
Lemma~\ref{lemma.p.non-decreasing}, $u(x,t)$ is non-increasing in $t$ for
$t\geq x^2/\alpha^2$.  This implies \ref{i.p0.iii}.
\end{proof}

\begin{theorem}
\label{u.uniform.convergence}
Let $(u,p)$ be a weak solution to \eqref{e.original} with
$u^*<\Psi(\alpha)$.  Assume that $p$ satisfies condition \textup{(P)}
and that there exists $\gamma \geq 0$ such that
\begin{equation}
  \lim_{x \to \infty}
  x \int^{\infty}_x p^*(\xi) \, \d \xi = \gamma \,.
  \label{e.uniform-convergence-assumption}
\end{equation}
Then $u$ converges uniformly to $\Phi_\gamma$.  Furthermore,
$\Phi_\gamma(\alpha)=u^*$ if $\gamma>0$ and $0<\Phi_0(\alpha)\le u^*$
if $\gamma=0$.
\end{theorem}

\begin{proof}
We shall show that for every $\eta>\alpha$ there exists $y$ such that
$p\equiv0$ on $D_{\eta,y}$.  Uniform convergence of $u$ to
$\Phi_\gamma$ is then a direct consequence of
Theorem~\ref{t.linear-asymptotics}.  To do so, assume the contrary,
i.e., that there exists $\eta_*>\alpha$ such that for every $y \in \R$
we have $p>0$ somewhere in $D_{\eta^*,y}$.  Due to
Lemma~\ref{p=0-on-D.infty}, this implies that there exists a sequence
$t_i\to\infty$ such that $u(x_i, t_i) \ge u^*$ with
$x_i=\eta_*\sqrt{t_i}$.  

We now claim that $p^*(x)=1$ for every $x \in (\alpha \sqrt{t_i}, x_i)$.
To prove the claim, fix $t_i$ and choose $X$ large enough such that
$\max_{t\in[0,t_i]}u(X,t)\le \psi(X,t_i)<u^*/2$.  Fix
$x \in (\alpha \sqrt t, x_i)$ and consider the cylinder
$U=(x, X)\times(0,t_i)$ with parabolic boundary $\Gamma$.  By the
parabolic maximum principle,
\begin{equation}
\label{p=1}
   \max_{\Gamma}u=\max_{\bar U}u\ge u(x_i,t_i)
\end{equation}
with equality only if $u$ is constant, which is incompatible with the
initial condition.  Hence, 
\begin{equation}
   \max_{t\in[0,t_i]} u(x,t) > u(x_i,t_i) \ge u^* \,.
\end{equation}
Since $p$ satisfies \eqref{e.hhmo-p-weak-alternative}, this implies
$p(x,t_i) = p^*(x) = 1$ as claimed.  Next, for $z_i=\alpha\sqrt{t_i}$,
we estimate
\begin{equation}
   z_i \int_{z_i}^\infty p^*(\xi)\,\d\xi
  \ge z_i\int_{z_i}^{x_i} p^*(\xi)\,\d\xi
  = z_i \, (x_i-z_i) = \eta_* \, (\eta_*-\alpha) \, t_i
  \to \infty
\end{equation}
as $i \to \infty$.  This contradicts
\eqref{e.uniform-convergence-assumption}.  We conclude that $p\equiv0$
on $D_{\eta_*,y}$ for some $y>0$.

To prove the final claim of the theorem, we note that
$\Phi_\gamma(\alpha)>u^*$ would imply the existence of a ring with
infinite width, which is impossible due to
Theorem~\ref{no.ring.infinite}.  Hence,
$0<\Phi_\gamma(\alpha)\le u^*$.  When $\gamma=0$, this is all that is
claimed.  So supposed that $\gamma>0$ and $\Phi_\gamma(\alpha)< u^*$.
Then Lemma~\ref{p=0-on-D.infty}\ref{i.p0.iii} implies that
$p(\xi,t)=p^*(\xi)=0$ for all $\xi$ big enough, say, when $\xi\ge R$,
and therefore
\begin{equation}
  x\int^{\infty}_x p^*(\xi) \, \d \xi = 0
\end{equation}
for $x\ge R$, contradicting \eqref{e.uniform-convergence-assumption}.
Hence, $\Phi_\gamma(\alpha)=u^*$ when $\gamma>0$.
\end{proof}

We now prove a result which provides a converse to
Theorem~\ref{u.uniform.convergence}.  We assume that a solution to
\eqref{e.original} has limit profile in $\eta$-$s$ coordinates and
conclude that this limit can only be the self-similar profile
$\Phi_\gamma(\eta)$ from \eqref{e.phi.gamma}.

\begin{theorem}
\label{limit.profile.th}
Let be $(u,p)$ a weak solution to \eqref{e.original} with
$u^*<\Psi(\alpha)$.  Assume that $p$ satisfies condition \textup{(P)}
and that for a.e.\ $\eta \geq 0$ the limit
\begin{equation}
  V(\eta)=\lim_{t\to\infty}u(\eta\sqrt t,t)
  = \lim_{s \to \infty} v(\eta,s)
  \label{e.limit-assumption}
\end{equation}
exists.  Then the limits
\begin{equation}
  \gamma = \lim_{x \to \infty} \frac1x \int_0^x \xi^2 \, p^*(\xi) \, \d\xi
\end{equation}
and
\begin{equation}
  \gamma = \lim_{x \to \infty} x \int_x^\infty p^*(\xi) \, \d\xi
\end{equation}
exist and are equal, $V(\eta)=\Phi_\gamma(\eta)$, and $u$ converges uniformly to
$\Phi_\gamma$.  Further, $\Phi_\gamma(\alpha)=u^*$ if $\gamma>0$ and
$0<\Phi_\gamma\le u^*$ if $\gamma=0$.
\end{theorem}

\begin{proof}
Write $U_V$ to denote the domain of definition of $V$, change the
coordinate system into $\eta=x/\sqrt t$ and $s=\sqrt t$, and set
$w=v-\Psi$ and $W=V-\Psi$.  As detailed in Appendix~\ref{a.weak}, the
weak formulation of the HHMO-model in these similarity variables can
be stated as
\begin{align}
  A(S; S_0, f)
  & = \frac{S_0}S\int_\R w(\eta, S_0) \, f(\eta) \, \d\eta\notag 
    - \int_\R w(\eta, S) \, f(\eta) \, \d\eta\notag\\
  & \quad
    - \frac1S\int_{S_0}^S\int_\R \eta \, w \, f_\eta \, \d \eta \, \d s
    - \frac2{S}\int_{S_0}^S\int_\R w_\eta \, f_\eta\, \d \eta \, \d s
  \label{A.equation.eta-s}
\end{align}
for all $0<S_0<S$ and $f \in H^1(\R)$ with compact support, where
\begin{equation}
  A(S; S_0, f)
  = \frac2S \int_{S_0}^S \int_\R s^2 \, q(\eta, s) \,
    v(\eta, s) \, f(\eta) \, \d \eta \, \d s \,.
  \label{e.Adef.eta-s}
\end{equation}
Writing
\begin{align}
  A(S; S_0, f)
  & = \frac{S_0}S\int_\R w(\eta, S_0) \, f(\eta) \, \d\eta 
    - \int_\R w(\eta, S) \, f(\eta) \, \d\eta\notag\\
  & \quad 
    - \frac1S\int_{S_0}^S\int_\R \eta \, w \,  f_\eta \,\d \eta \, \d s
    + \frac2{S}\int_{S_0}^S\int_\R w \, f_{\eta\eta}\, \d \eta \, \d s \,,
  \label{A.equation.temporary.eta-s}
\end{align}
we observe that the limit $S\to\infty$ exists for each term on the
right of \eqref{A.equation.temporary.eta-s}, so that
$\lim_{S\to\infty} A(S; S_0, f)$ exists for $S_0$ and $f$ fixed.
Moreover, for every $b>0$ fixed, definition \eqref{e.Adef.eta-s}
implies that $A(S;S_0,f)$ is bounded uniformly for all $S \geq S_0$
and $f\in L^1$ that satisfy $0\le f\le \I_{[-b,b]}$.  Indeed, if
$g \geq \I_{[-b,b]}$ is smooth with compact support, then
\begin{equation}
  A(S; S_0, f)\le A(S; S_0, g)\le \sup_{S\ge S_0} A(S; S_0, g) < \infty
\end{equation}
since $\lim_{S\to\infty} A(S; S_0, g)$ exists.

By Lemma~\ref{lemma.p.non-decreasing}, $u-\psi$ is non-increasing in
time $t$ for $x$ fixed. This implies that, in $\eta$-$s$ coordinates,
for $\eta_1, \eta_2 \in U_V$ with $0<\eta_1 < \eta_2$,
\begin{equation}
  W(\eta_1)
  = \lim_{s \to \infty} w(\eta_1,s)
  \leq \lim_{s \to \infty} w(\eta_2,s)
  = W(\eta_2) \,,
  \label{e.lim-comparison.eta-s}
\end{equation}
and for any fixed $\eta\in(\eta_1,\eta_2)$ 
\begin{align}
  W(\eta_1)
  & = \lim_{s \to \infty} w(\eta_1,s)
    \leq \liminf_{s \to \infty} w(\eta,s) \notag\\
  & \leq \limsup_{s \to \infty} w(\eta,s)
    \leq \lim_{s \to \infty} w(\eta_2,s) = W(\eta_2) \,.
  \label{e.lim-comparison.middle.eta-s}
\end{align}

By Lemma~\ref{l.u0}, $V(0) = 0$, so that
\begin{equation}
  W(0)=V(0)-\Psi(0)=-\Psi(0)\le V(\eta)-\Psi(\eta)=W(\eta)
\end{equation}
for all $\eta\in U_V$.  Altogether, we find that $W=V-\Psi$ is
non-decreasing on $U_V$.

Now we will show that $W$ is locally Lipschitz continuous on $U_V$.
Fix $b>0$.  For every $\eta_0\in[0,b]$, take the family of compactly
supported test functions $f_\eps(\eta)$ whose derivative is given by
\begin{equation}
  f_\eps'(\eta)
  = \begin{dcases}
      \eps^{-1} & \text{for } \eta \in[-\eps,0] \,, \\
      - \eps^{-1} & \text{for } \eta\in[\eta_0,\eta_0+\eps] \,, \\
      0 & \text{ otherwise} \,.
    \end{dcases}
\end{equation}
We insert $f_\eps$ into \eqref{A.equation.eta-s} and let $\eps\searrow0$.
Clearly,
\begin{equation}
  \lim_{\eps\searrow0} A(S;S_0,f_\eps)=A(S;S_0,\I_{[0,\eta_0]}(\eta))
\end{equation}
and
\begin{equation}
  \int_\R w(\eta, s)f(\eta)\,\d \eta
  \to \int_0^{\eta_0} w(\eta, s)\,\d \eta \,.
\end{equation}
Moreover,
\begin{equation}
  \int_\R \eta \, w(\eta,s) \, f_\eps'(\eta)\, \d\eta
  \to - \eta_0 \, w(\eta_0,s)
\end{equation}
and
\begin{equation}
  \int_\R w_\eta(\eta, s) \, f_\eps'(\eta)\, \d\eta
  \to w_\eta(0,s) - w_\eta(\eta_0,s) = - w_\eta(\eta_0,s) \,.
\end{equation}
(Recall that $w_\eta$ is space-time continuous due to the definition
of weak solution.)  Altogether, we find that \eqref{A.equation.eta-s}
converges to
\begin{align}
  A(S; S_0, \I_{[0,\eta_0]}(\eta))
  & = \frac{S_0}S\int_0^{\eta_0} w(\eta, S_0)\, \d\eta 
     - \int_0^{\eta_0} w(\eta, S)\, \d\eta\notag\\
  & \quad 
    + \frac{\eta_0}S\int_{S_0}^Sw(\eta_0, s)\, \d s
    + \frac2{S}\int_{S_0}^S
      w_\eta(\eta_0, s) \, \d s \,.
\end{align}
Noting that $0\le \I_{[0,\eta_0]}\le f_\eps \le \I_{[-1,b+1]}$ for
$0<\eps\leq 1$, we see that the left hand side is bounded uniformly
for all $\eta_0\in[0,b]$ and $S \geq S_0$.  By direct inspection, so
are the first three terms on the right hand side.  We conclude
that
\begin{equation}
  \frac1S \int_{S_0}^S w_\eta(\eta_0,s) \, \d s \leq C_b
\end{equation}
for some constant $C_b$ independent of $\eta_0\in[0,b]$ and
$S\ge S_0$.  Then, for any pair $\eta_1,\eta_2 \in U_V\cap[0,b]$ with
$\eta_1<\eta_2$,
\begin{align}
  0 \leq W(\eta_2)-W(\eta_1)
  & = \lim_{S\to\infty} \frac1S \int_{S_0}^S w(\eta_2,s) \, \d s
      - \lim_{S\to\infty} \frac1S\int_{S_0}^S w(\eta_1,s) \, \d s
      \notag \\
  & = \lim_{S\to\infty} \frac1S \int_{S_0}^S \int_{\eta_1}^{\eta_2}
        w_\eta(\eta,s) \, \d\eta \, \d s
      \notag \\
  & = \lim_{S\to\infty} \int_{\eta_1}^{\eta_2}
        \frac1S \int_{S_0}^S w_\eta(\eta,s) \, \d s \, \d\eta
      \notag \\
  & \le C_b \, \abs{\eta_2-\eta_1} \,.
\end{align}
Due to \eqref{e.lim-comparison.middle.eta-s}, we conclude that $W$ is
locally Lipschitz continuous, defined on $U_V=\R_+$, and
non-decreasing.  In particular, $V(\alpha)$ is well-defined and
strictly positive.  To see the latter, suppose the contrary, i.e.,
that $V(\alpha)=0$.  Then Lemma~\ref{p=0-on-D.infty}\ref{i.p0.iii}
implies that $p(x,t)=0$ for all $x$ large enough. It follows that we
can take $\gamma=0$ in Theorem~\ref{t.linear-asymptotics} to conclude
that $V=\Phi_0$, contradicting $V(\alpha)=0$.

Since $V(\alpha)>0$, there is a neighborhood
$I=(\eta_0,\eta_1)\subset (0,\alpha)$ such that $V>\frac12V(\alpha)>0$
on $I$.  Further, set
\begin{subequations}
\begin{gather}
  v_+(\eta;S_0) = \sup_{s\ge S_0}v(\eta,s) \,, \\
  v_-(\eta;S_0) = \inf_{s\ge S_0}v(\eta,s) 
\end{gather}
\end{subequations}
and choose $S_0^*$ large enough such that
$v_-(\eta_0, S_0^*)>\frac12V(\eta_0)>0$. Since $u-\psi$ is
non-increasing in time in $x$-$t$ coordinates and $\Psi$ is constant
on $I$, we have $v_+(\eta;S_0)\ge v_-(\eta;S_0)\ge\frac12V(\eta_0)$
for all $S_0\ge S_0^*$ and $\eta \in I$.  Take $g \in H^1(\R)$ with
$\supp g\subset I$.  Noting that, due to \eqref{p.property},
$q(\eta,s) = p^*(\eta s)$, we estimate
\begin{align}
  A(S; S_0, g)
  & = \int_{\eta_0}^{\eta_1} \frac2S\int_{S_0}^{S}
	s^2 \, p^*(\eta s) \, v(\eta,s) \,
        g(\eta) \, \d s \, \d \eta
      \notag \\
  & \le \int_{\eta_0}^{\eta_1} g(\eta) \, v_+(\eta; S_0) \,
	\frac2S\int_{S_0}^{S} s^2 \, p^*(\eta s) \,\d s \, \d \eta
      \notag \\
  & = \int_{\eta_0}^{\eta_1} \frac{g(\eta)}{\eta^3} \,  
	v_+(\eta; S_0) \, \frac2S\int_{S_0\eta}^{S\eta}
        \xi^2 \, p^*(\xi)\,
        \d\xi \, \d\eta
      \notag \\
  & \le \int_{\eta_0}^{\eta_1} \frac{2g(\eta)}{\eta^3} \,  
	v_+(\eta; S_0) \, \d \eta \,
        \frac1S\int_{0}^{S\eta_1} \xi^2 \, p^*(\xi)\, \d\xi 
\end{align}
where, in the second equality, we have used the change of variables 
$\xi=s\eta$.  Taking $\liminf_{S \to \infty}$, we infer that
\begin{equation}
  \lim_{S\to\infty} A(S; S_0, g)
  \le \gamma^- \, \eta_1
      \int_{\eta_0}^{\eta_1} \frac{2g(\eta)}{\eta^3} \,
        v_+(\eta;S_0) \, \d\eta \,,
  \label{e.ubound.eta-s}
\end{equation}
where 
\begin{equation}
  \gamma^-
  = \liminf_{S\to\infty} \frac1S \int_0^S \xi^2 \,
      p^*(\xi) \, \d\xi \,. 
\end{equation}
Similarly,
\begin{align}
  A(S; S_0, g)
  & \ge \int_{\eta_0}^{\eta_1} g(\eta)\, v_-(\eta;S_0)\,
	\frac2S\int_{S_0}^Ss^2 \, 
	p^*(\eta s) \, \d s \, \d\eta
      \notag \\
  & = \int_{\eta_0}^{\eta_1} \frac{g(\eta)}{\eta^3} \,
        v_-(\eta;S_0) \, 
        \frac2S \int_{S_0\eta}^{S\eta}
        \xi^2 \, p^*(\xi) \, \d\xi \, \d\eta 
      \notag \\
  & \ge \int_{\eta_0}^{\eta_1} \frac{2g(\eta)}{\eta^3} \,
        v_-(\eta;S_0) \, \d\eta \, 
        \frac1S \int_{S_0\eta_1}^{S\eta_0}
        \xi^2 \, p^*(\xi) \, \d\xi \,,
\end{align}
so that
\begin{equation}
  \lim_{S\to\infty} A(S; S_0, g)
  \ge \gamma^+ \, \eta_0\int_{\eta_0}^{\eta_1}
      \frac{2g(\eta)}{\eta^3} \, v_-(\eta;S_0) \, \d\eta
  \label{e.lbound.eta-s}
\end{equation}
with
\begin{equation}
  \gamma^+
  = \limsup_{S\to+\infty} \frac1S \int_0^S \xi^2 \,
      p^*(\xi) \, \d\xi
  \geq \gamma_- \,.
\end{equation}
Equation \eqref{e.lbound.eta-s} also implies that $\gamma^+<\infty$.
Since the bounds \eqref{e.ubound.eta-s} and \eqref{e.lbound.eta-s} are
valid for arbitrary $S_0\ge S_0^*$, we can now let $S_0 \to \infty$,
so that
\begin{equation}
  \gamma^+ \, \eta_0 \int_{\eta_0}^{\eta_1}
    \frac{2g(\eta)}{\eta^3} \, V(\eta) \, \d\eta
  \le \gamma^- \, \eta_1 \int_{\eta_0}^{\eta_1}
    \frac{2g(\eta)}{\eta^3} \, V(\eta) \, \d\eta \,.
\end{equation}
Since $V>0$ on $I$, we can divide out the integral to conclude that
$\gamma_+ \, \eta_0 \le \gamma_- \, \eta_1$.  Further, we can take
$\eta_0$ and $\eta_1$ arbitrary close to each other by taking a test
function $g$ with arbitrarily narrow support, so that
$\gamma_+=\gamma_-$ and both are equal to
\begin{equation}
  \gamma = \lim_{S\to\infty} \frac1S \int_0^S \xi^2 \,
    p^*(\xi) \, \d\xi < \infty \,.
  \label{e.gamma}
\end{equation}

To proceed, we define
\begin{equation}
  \Gamma(x)
  = x\int_x^\infty p^*(\xi) \, \d \xi
\end{equation}
as in the proof of Theorem~\ref{t.linear-asymptotics}, introduce its
average
\begin{equation}
  \bar \Gamma(x)
  = \frac1x \int_0^x \Gamma(\xi) \, \d \xi \,,
\end{equation}
and set
\begin{equation}
  h(x)
  = \frac1x \int_0^x \xi^2 \, p^*(\xi) \, \d \xi \,.
\end{equation}
In \eqref{e.gamma}, we have already shown that $h(x) \to \gamma$ as
$x \to \infty$.  It remains to prove that $\Gamma(x) \to \gamma$ as
well.  We first note that $p^*$ is integrable so that $\Gamma$ is
well-defined.  To see this, we write
\begin{equation}
  \int_1^x p^*(\xi) \, \d \xi
  = \int_1^x \frac1{\xi^2} \, \xi^2 \, p^*(\xi) \, \d \xi
  = \frac{h(x)}x - h(1)
    + 2 \int_1^x \frac{h(\xi)}{\xi^2} \, \d \xi \,,
\end{equation}
where we have integrated by parts, noting that $x\, h(x)$ is an
anti-derivative of $x^2\, p^*(x)$.  As $h(x)$ converges and $p^*$ is
non-negative, $p^*$ is integrable on $\R_+$.

Next, by direct calculation,
\begin{equation}
  \xi^2 \, p^*(\xi) = \Gamma(\xi) - \xi \, \Gamma'(\xi) \,.
\end{equation}
Inserting this expression into the definition of $h$ and integrating
by parts, we find that
\begin{equation}
  h(x) = 2 \, \bar \Gamma(x) - \Gamma(x)
  \label{e.limit-identity1}
\end{equation}
First, divide \eqref{e.limit-identity1} by $x$ and observe that
$h(x)/x$ and $\Gamma(x)/x$ both converge to zero as $x\to\infty$.
Consequently,
\begin{equation}
  \lim_{x \to \infty} \frac{\bar \Gamma(x)}x = 0 \,.
  \label{e.bargammalimit}
\end{equation}
Second, note that \eqref{e.limit-identity1} can be written in the form
\begin{equation}
 \frac{h(x)}{x^2} = -\frac{\d}{\d x} \frac{\bar \Gamma(x)}x \,.
\end{equation}
Integrating from $x$ to $\infty$ and using \eqref{e.bargammalimit}, we
find that
\begin{equation}
  \bar \Gamma(x)
  = x \int_x^\infty \frac{h(\xi)}{\xi^2} \, \d \xi \,.
\end{equation}
Since $h(x) \to \gamma$, this expression converges to $\gamma$ by
l'H\^opital's rule.  Thus, by \eqref{e.limit-identity1},
$\Gamma(x) \to \gamma$ as well.  We recall that, due to
\cite[Lemma~3.5]{HilhorstHM:2009:MathematicalSO}, $p$ has at least one
non-degenerate precipitation region, so $p$ is non-zero.  Hence, we
can finally apply Theorem~\ref{u.uniform.convergence} which asserts
uniform convergence of $u$ to $\Phi_\gamma$.
\end{proof}

We summarize the results of this section in the following theorem.

\begin{theorem}
\label{convergence.conclusion}
Let $(u,p)$ be a weak solution to \eqref{e.original} with
$u^*<\Psi(\alpha)$.  Assume that $p$ satisfies condition \textup{(P)}.
Then the following statements are equivalent.
\begin{enumerate}[label={\upshape(\roman*)}]
\item\label{equiv.0}
$\displaystyle \lim_{x \to \infty} \frac1x \int_0^x \xi^2 \, p^*(\xi)
\, \d \xi = \gamma$,

\item\label{equiv.i}
$\displaystyle \lim_{x \to \infty} x \int^{\infty\vphantom\int}_x
p^*(\xi) \, \d \xi = \gamma$,

\item\label{equiv.ii} $u$ converges uniformly to $\Phi_\gamma$ with
$\Phi_\gamma(\alpha)=u^*$ if $\gamma>0$ and
$0<\Phi_\gamma(\alpha)\le u^*$ if $\gamma=0$,

\item\label{equiv.iii} $u$ converges to some limit profile $V$
pointwise a.e.\ in $\eta$-$s$ coordinates.
\end{enumerate}
\end{theorem}

\begin{proof}
Statement \ref{equiv.i} implies \ref{equiv.ii} by
Theorem~\ref{u.uniform.convergence}, and \ref{equiv.ii} trivially
implies \ref{equiv.iii}.  Conversely, \ref{equiv.iii} implies
\ref{equiv.0} and \ref{equiv.0} implies \ref{equiv.i} by
Theorem~\ref{limit.profile.th} and its proof.
\end{proof}

\appendix
\section{Numerical scheme for the HHMO-model}
\label{a.scheme}

To solve the model \eqref{e.original} numerically, it is convenient to
define $w=v-\Psi$, where $v$ satisfies the HHMO-model in $\eta$-$s$
coordinates, equation \eqref{e.v}, and $\Psi$ is the self-similar
solution without precipitation from Section~\ref{s.without}.  Then $w$
solves the equation
\begin{subequations}
\begin{gather}
  \label{discr}
  s \, w_s - \eta \, w_\eta
  = 2 \,w_{\eta\eta} - 2 \, s^2 \, q[\eta,s] \, (\Psi + w) \,, \\
  w_\eta(0,s) = 0 \quad \text{for } s > 0 \,, \\
  w(\eta,s) \to 0 \quad \text{as } \eta \to \infty \text{ for } s > 0 \,.
\end{gather}
\end{subequations}

We take $N$ cells on the interval $[0,\alpha]$ of width
$\Delta \eta= \alpha/N$ and extend the domain of computation to the
right up to a total of $N_\full=6N$ grid cells.  (The factor $6$ is
empirical, but works robustly due to the rapid decay of the
concentration field.) Thus, the spatial nodes are given by
$\eta_i = i \Delta \eta$ for $i = 0, \dots, N_\full-1$.

In time, we run $M$ steps up to a total time $s=S$, so that
$\Delta s = S/M$.  Setting $s_j = j \Delta s$, we write $w_i^j$ to
denote the numerical approximation to $w(\eta_i,s_j)$, $q_i^j$ to
denote the numerical approximation to $q(\eta_i,s_j)$, and set
$\Psi_i = \Psi(\eta_i)$.  We use implicit first order timestepping, a
first order upwind finite difference for the advection term, and the
standard second order finite difference approximation for the
Laplacian, i.e.
\begin{subequations}
\begin{gather}
  \label{w.discr}
  w_{\eta\eta}(\eta_i, s_j)
  \approx \frac{w^j_{i+1}-2w^j_i+w^j_{i-1}}{\Delta\eta^2} \,, \\
  w_{\eta}(\eta_i, s_j)
  \approx \frac{w^j_{i+1}-w^j_i}{\Delta\eta} \,, \\
  w_s(\eta_i, s_j)
  \approx \frac{w^j_i-w^{j-1}_i}{\Delta s} \,.
\end{gather}
The Neumann boundary condition at $\eta=0$ is approximated by
\begin{equation}
  w_{-1}^j=w_0^j
\end{equation}
and the decay condition is approximated by the
homogeneous Dirichlet condition
\begin{equation}
  w_{6N+1}^j=0 \,.
\end{equation}
\end{subequations}
The precipitation term is treated explicitly.  Altogether, this leads
to the system of equations $A^j w^{j} = b^{j-1}$ where $A^j$ is a
tridiagonal matrix with coefficients
\begin{subequations}
\begin{gather}
  a_{i,i}^j = j + i + 4/\Delta \eta^2
  \qquad \text{for } i = 0, \dots, N_\full -1 \,, \\
  a_{i,i-1}^j = - 2/\Delta \eta^2
  \qquad \text{for } i = 1, \dots, N_\full - 1 \,, \\
  a_{i,i+1}^j = - i - 2/\Delta \eta^2
  \qquad \text{for } i = 1, \dots, N_\full - 2 \,, \\
  a_{0,1}^j = -4/\Delta \eta^2  \,,
\end{gather}
and $b^{j-1}$ is a vector with coefficients
\begin{gather}
  b_i^{j-1}
  = j \, w_i^{j-1}
    - 2 \, j^2 \, \Delta s^2 \, q_i^{j-1} \, (\Psi_i + w_i^{j-1})
  \qquad \text{for } i = 0, \dots, N_\full -1 \,.
\end{gather}
\end{subequations}

It remains to determine an expression for the $q_i^j$.  Note that once
$q(\eta_0,s_0)$ equals one at some point $(\eta_0,s_0)$, it will
remain one along the characteristic curve $\eta s = \eta_0 s_0$ for
all $s \geq s_0$.  This gives rise to the following simple consistent
transport scheme.

We first consider spatial indices $i \geq N$.  In this region, we
observe that whenever $u_i^j > u^*$, the maximum principle for the
continuum problem implies that $u$ exceeds the precipitation threshold
on some curve contained in the region $\{s \leq s_j\}$ which connects
the point $(\eta_i,s_j)$ with the line $\eta=\alpha$.  This implies
that $q_k^j = 1$ for all $N \leq k \leq i$.  Consequently, we only
need to track the largest index $I^j$ where precipitation takes place
and set $q_k^j=1$ for $k = N, \dots, I^j$.  To do so, observe that
precipitation takes place either when $u$ exceeds the threshold, or
when a cell lies on a characteristic curve where precipitation has
taken place at the previous time step.  This leads to the the
expression
\begin{equation}
  I^j
  = \max \bigl\{
           \max \{ k \colon u_k^j > u^* \},
           \lfloor I^{j-1} \, (j-1)/j \rfloor
         \bigr\} \,.
\end{equation}

Second, for spatial indices $i<N$ corresponding to $\eta < \alpha$, we
only need to transport the values of the precipitation function along
the characteristic curves.  We note that the characteristic curves
define a map from the temporal interval $[0,s]$ at $\eta=\alpha$ to
the spatial interval $[0,\alpha]$ at time $s=j \Delta s$.  This map
scales each grid cell by a factor $N/j$.  We distinguish two
sub-cases.  For fixed time index $j \leq N$, a temporal cell is mapped
onto at least one full spatial cell.  Thus, we can use a simple
backward lookup as follows.  Let
\begin{equation}
  \Jc(i;j) = \lfloor i \tfrac{j}N \rfloor
  \label{e.ji}
\end{equation}
be the time index in the past that corresponds best to spatial index
$i$.  Then we set
\begin{equation}
  q_i^j = q_N^{\Jc(i;j)} \,.
\end{equation}

For a fixed time index $j>N$, we do a forward mapping, i.e., we define
the inverse function to \eqref{e.ji},
\begin{equation}
  \Ic(k;j) = \lceil k \tfrac{N}j \rceil \,,
\label{e.ij}
\end{equation}
which represents the spatial index that the cell with past time index
$k$ and spatial index $N$ has moved to, and set
\begin{equation}
  \label{f.ij}
  q_i^j = \frac{N}j \sum_{\Ic(k;j)=i} q_N^k \,.
\end{equation}
Note that this expression can yield values for $q_i^j$ outside of the
unit interval, which is not a problem as the integral over the entire
interval is represented correctly.  To implement this efficiently in
code, we keep a running sum
\begin{gather}
  Q_j = \sum_{k = 0}^j q_N^k
\end{gather}
which can be updated incrementally and write
\begin{gather}
  q_i^j = \frac{N}j \, (Q_{\Jc(i+1;j)} - Q_{\Jc(i;j)}) \,.
\end{gather}
This expression is equivalent to \eqref{f.ij}.

\section{Weak formulation in similarity coordinates}
\label{a.weak}

In the following, we provide the details of changing to the weak
formulation in similarity coordinates.  By formal computation, for an
arbitrary function $h$,
\begin{subequations}
\label{eta-s.trans}
\begin{gather}
   h_t = -\frac12 \, h_\eta \, \frac x{t^{3/2}}
         +\frac12 \, h_s \, \frac 1{\sqrt{t}}
       = -\frac\eta{2s^2} \, h_\eta + \frac1{2s} \, h_s \,, \\
   h_x = h_\eta \, \frac1{\sqrt t} = \frac1s \, h_\eta \,,
\end{gather}
and the Jacobian of the change of variables reads
\begin{equation}
  \frac{\partial(x,t)}{\partial(\eta,s)}=
	\begin{vmatrix}
	   s & \eta \\
	   0 & 2s\\
	\end{vmatrix}
	= 2s^2 \,.
\end{equation}
\end{subequations}
Now, take the weak formulation of the HHMO-model
\eqref{weak.sol.def.eq}, and replace the partial derivatives
$\varphi_t$, $\varphi_x$, and $(u-\psi)_x$ in terms of $\varphi_s$,
$\varphi_\eta$, and $(v-\Psi)_{\eta}$ according to
\eqref{eta-s.trans}.  This yields
\begin{equation}
  \label{weak.sol.def.eq.1}
  \int_0^{\sqrt T} \int_\R (s \, \varphi_s - \eta \, \varphi_\eta) \,
    (v-\Psi) \, \d \eta \, \d s
  = 2 \int_0^{\sqrt T} \int_\R
    \bigl( 
      (v-\Psi)_{\eta} \, \varphi_{\eta}  +s^2 \, q \, v \, \varphi
    \bigr) \, \d \eta \, \d s \,.
\end{equation}
We can extend the class of admissible test function to product test
functions of the form
\begin{equation}
  \varphi(\eta,s) = f(\eta) \, \chi(s) 
\end{equation}
where $f\in H^1(\R)$ with compact support and $\chi\in H^1(\R)$ with
compact support in $(0,\infty)$ by density.  Inserting $\varphi$ into
\eqref{weak.sol.def.eq.1} and setting $w = v-\Psi$, we obtain
\begin{equation}
  \int_0^{\sqrt T} \int_\R (s \, f\chi_s - \eta \, f_\eta\,\chi) \,
    w \, \d \eta \, \d s
  = 2 \int_0^{\sqrt T} \int_\R
      \bigl(
        w_\eta \, f_\eta \, \chi + s^2 \, q \, v \, f\, \chi
      \bigr) \, \d \eta \, \d s \,.
 \label{e.weak2}
\end{equation}
Fix $0<S_0<S<\sqrt T$ and let $\chi\in H^1(\R_+)$ be the test function
with derivative
\begin{equation}
  \chi_s(s)
  = \begin{dcases}
      \eps^{-1} & \text{for } s \in[-\eps+S_0,S_0] \,, \\
      - \eps^{-1} & \text{for } s \in[S,S+\eps] \,, \\
      0 & \text{ otherwise} \,.
    \end{dcases}
\end{equation}
Finally, insert this expression into \eqref{e.weak2} and let
$\eps\searrow0$.  This implies that $w$ is a weak solution to the
HHMO-model in similarity variables if
\begin{align}
  S_0 \int_\R w(\eta, S_0) \, f(\eta) \, \d\eta
  & - S \int_\R w(\eta, S) \, f(\eta) \, \d\eta
    - \int_{S_0}^S\int_\R \eta \, f_\eta \, w \, \d \eta \, \d s \notag \\
  & = 2 \int_{S_0}^S\int_\R
      \bigl(
        w_\eta \, f_\eta + s^2 \, q \, v \, f
      \bigr) \, \d \eta \, \d s 
   \label{weak.equation.w.eta-s}
\end{align}
for all $0<S_0<S<\sqrt T$ and $f \in H^1(\R)$ with compact support. 

\section*{Acknowledgments}

We thank Danielle Hilhorst and Arndt Scheel for interesting
discussions, and an anonymous referee for their careful attention to
detail and helpful comments.  This work was funded through German
Research Foundation (DFG) grant OL 155/5-1.  Additional funding was
received via the Collaborative Research Center TRR 181 ``Energy
Transfers in Atmosphere and Ocean'', also funded by the DFG under
project number 274762653.

\bibliographystyle{siam}
\bibliography{liesegang}

\end{document}